\documentclass[11pt,a4paper,reqno]{amsart}
\usepackage[T1]{fontenc}
\usepackage{csquotes}
\usepackage{amsmath, amsthm, amsfonts,geometry, amssymb,color}
\usepackage{a4wide,amsfonts,amsmath,latexsym,amssymb,euscript,graphicx,units,mathrsfs,bbold}  \usepackage{multicol} 
\usepackage{comment}
\usepackage{bbm}
\usepackage{theoremref}

  \def\talpha{t_{\alpha}}

 \newcommand{\ex}[1]{\E\left[#1\right]}
 \newcommand{\sigmaone}[1]{\Sigma^{(1)}}
  \newcommand{\sigmatwo}[1]{\Sigma^{(2)}}
 \usepackage[shortlabels]{enumitem}%to itemize with no margin(no indent)
   \usepackage{hyperref}
\hypersetup{
    colorlinks=true,
    linkcolor=blue,
    filecolor=magenta,      
    urlcolor=cyan,
       pdftitle={Overleaf Example},
    pdfpagemode=FullScreen,
    }
\usepackage{kantlipsum}

\def\HH{\mathfrak{H}}

\def\R{\mathbb{R}}
\def\D{\mathbb{D}}
\def\E{\mathrm{E}}
\def\P{\mathrm{P}}

\def\N{\mathbb{N}}

\def\Var{{\mathrm{Var}}}
\def\dom{{\mathrm{Dom}}}

  \newcommand{\e}{\varepsilon}
\def\x{\mathcal{X}}
\def\y{\mathcal{Y}}
 \def\lip{\textrm{Lip}} 
 \def\H{\mathcal{H}}
  
  \newtheorem{theorem}{Theorem}[section]

\newtheorem{lemma}[theorem]{Lemma}
 \newtheorem{proposition}[theorem]{Proposition}

\numberwithin{equation}{section}
\begin{document}
   
   \title[Convergence of Densities of Spatial Averages of Stochastic Heat Equation]{Convergence of Densities of Spatial Averages of Stochastic Heat Equation}

 \author[S. Kuzgun]{Sefika Kuzgun}
\address{University of Kansas, Department of Mathematics, USA}
\email{sefika.kuzgun@ku.edu}

\author[D. Nualart]{David Nualart} \thanks{%
	The work by D. Nualart has been supported   by the  NSF grants DMS-2054735}
\address{University of Kansas, Department of Mathematics, USA}
\email{nualart@ku.edu}

\begin{abstract} In this paper, we consider the one-dimensional stochastic heat equation driven by a space time white noise.  
In two different scenarios:  {\it (i)}  initial condition $u_0=1$ and general nonlinear coefficient $\sigma$  and {\it (ii)}:  initial condition $u_0=\delta_0$ and $\sigma(x)=x$ (Parabolic Anderson Model), we establish rates of convergence for the uniform distance between  the density of (renormalized) spatial averages and the standard normal density.   
These results are based on the combination of Stein method for normal approximations and Malliavin calculus techniques. A key ingredient in Case (i)  is a new estimate on the $L^p$-norm of the second Malliavin derivative.


\end{abstract}

\maketitle

%{\bf Variable convention}\\
 %$t,x$ is the domain of $u$\\
 %$s,y$ 1st derivative
 %$\tau, \xi$ integral region\\
 %$r,z$ 2nd derivative variables\\
 %$\xi$ is fourier side\\
 %$w$ convolution variable\\
 %$\eta$'s are the pinned variables

\section{Introduction}
Consider the  one-dimensional stochastic heat equation
 \begin{equation} \label{SHE}
  \frac{\partial u}{\partial t}= \frac{1}{2}  \frac{\partial^2 u}{\partial x^2}+ \sigma(u) \Dot{W},  \qquad x\in \R, \,\,   t>0, 
\end{equation} 
with initial condition $u(0,x)=u_0(x)$, where $\Dot{W}$  is a space-time white noise.  This is to say, informally, 
that $\Dot{W}=\{ \Dot{W}(t,x): (t,x) \in \R_+\times \R\}$ is a centered Gaussian random field with covariance
\begin{align*}
    \E \left(\Dot{W}(t_1,x_1)\Dot{W}(t_2,x_2)\right)=\delta_0(t_1-t_2)\delta_0(x_1-x_2),
\end{align*}
where $\delta_0$ is the Dirac delta measure at zero.  

The existence and uniqueness of a mild solution $u(t,x)$ to  equation \eqref{SHE} has been  proved by  Chen and Dalang in \cite{ChDa} (see also the lecture notes by Walsh \cite{Wa} in the case where $u_0$ is a bounded function), assuming that $\sigma$ is Lipschitz and  $u_0$ is a signed measure that satisfies the following  integrability condition for any $t>0$
\begin{equation} \label{int}
\int_\R   |u_0|(dx) p_t(x)  <\infty.
\end{equation}
Here and  along the paper we will make use of the notation $p_t(x):=\frac{1}{\sqrt{2\pi t}}e^{-x^2/2t}$ for $t>0$ and $x\in \R$.
 
 We are interested in the asymptotic behavior of the spatial averages of the solution  to equation \eqref{SHE}  in the following two particular cases:
 
 \medskip
 \noindent
 {\bf  Case 1:}  $u_0\equiv 1 $ and  $\sigma:\R \to \R$ is a Lipschitz function such that $\sigma(1) \not =0$.   

 \medskip
 \noindent
 {\bf Case 2:} $u_0= \delta_0 $ and $\sigma(x)=x$.
 
We observe that, for any $t> 0$,   in Case 1,  the process $\{ u(t,x): x\in \R\}$ is stationary and in Case 2, the process $\{ U(t,x): x\in \R\}$, where 
 \begin{align*}
    U(t,x):=\frac{u(t,x)}{p_t(x)}, 
  \end{align*}
is also stationary (see Amir, Corwin and Quastel \cite{AmCoQu}).

Fix $R>0$ and consider the corresponding centered and normalized  spatial averages defined by
\begin{align}
  \label{FR}  F_{R,t}:=\frac{1}{\sigma_{R,t}}\left(\int_{-R}^R u(t,x)dx-2R\right),  \text{ where  } \sigma^2_{R,t}:=\Var\left(\int_{-R}^R u(t,x)dx\right) 
  \end{align}
in Case 1,  and 
   \begin{align}  \label{GR} 
G_{R,t}:=\frac{1}{\Sigma_{R,t}}\left(\int_{-R}^R U(t,x)dx-2R\right),  \text{ where  }  \Sigma^2_{R,t}:=\Var\left(\int_{-R}^R U(t,x)dx\right)
\end{align}
in Case 2.
  
Huang, Nualart and Viitasaari \cite{HuNuVi} and Chen, Khoshnevisan, Nualart and Pu \cite{ChKhNuPu2},   studied  the limiting behavior of $F_{R,t}$ and $G_{R,t}$, respectively,   as $R$ tends to infinity. In these papers,   functional central limit theorems have been established. Moreover, by  the Malliavin-Stein approach, introduced by Nourdin and Peccati (see \cite{NoPe}), upper bounds for the total variation distance have been obtained. More precisely,   it has been proven that, for any fixed $t>0$ and for all $R\ge 1$, 
 \begin{equation}  \label{e82}
d_{\text{TV}}\left(F_{R,t}\,,N(0,1)\right) \leq \frac{C_{t}}{\sqrt{R}},
\end{equation} 
and 
\begin{equation}  \label{e83}
d_{\text{TV}}\left(G_{R,t}\,,N(0,1)\right) \leq \frac{C_t \sqrt{\log R} }{\sqrt{R}},
\end{equation} 
 where $d_{\text{TV}}$ denotes the total variation distance and $C_t$ is a constant depending on $t$.

The purpose of this paper is to derive upper bounds for the rate of convergence of the uniform distance of densities in the two cases above  mentioned. Upper bounds for the uniform distance of densities using techniques of Malliavin calculus were first derived by Hu, Lu and Nualart in \cite{HuLuNu}. Inspired by the methodology introduced in this reference, we have been able to obtain the following two main results.  In Case 1, we will make use of the following hypothesis on $\sigma$:

\medskip
\noindent
{\bf (H1)}:
  $\sigma:\R \to \R$ is a twice continuously differentiable function with $\sigma'$ bounded and $|\sigma''(x)| \le C(1+ |x|^m)$, for some $m>0$.

\begin{theorem}\label{densityshe}  In Case 1, let  $u=\left\{u(t,x): (t,x)\in  \R_+\times \R\right\}$  be the mild solution to the stochastic heat equation \eqref{SHE}.   Assume that $\sigma$ satisfies hypothesis {\bf (H1)}.    Suppose also  that for some $q>10$, $\ex{|\sigma(u(t,0))|^{-q}}<\infty$. 
Fix $t>0$ and 
  let $F_{R,t}$ be defined as in \eqref{FR}. Then, for all $R>0$,
\begin{align*}
  \sup_{x\in \R} |f_{F_{R,t}}(x)-\phi(x)| \leq \frac{C_t}{\sqrt{R}},   \end{align*} where $f_{F_{R,t}}$ and  $\phi$ are the densities of $F_{R,t}$ and $N(0,1)$, respectively.
\end{theorem}

We remark that condition  $\ex{|\sigma(u(t,0))|^{-q}}<\infty$  holds if $\sigma$ is bounded away from zero or if $|\sigma(x)|\le \Lambda |x|$ for all $x\in \R$  and for some constant $\Lambda>0$ (see, for instance,  \cite[Theorem 1.5]{ChHuNu}).
 
\begin{theorem}\label{densitypam} In Case 2, assume that the random field $u=\left\{u(t,x): (t,x)\in  \R_+\times \R\right\}$ solves the stochastic heat equation \begin{equation*}
  \frac{\partial u}{\partial t}= \frac{1}{2}  \frac{\partial^2 u}{\partial x^2}+ u \Dot{W},  \qquad x\in \R, \,\,  t>0, 
\end{equation*} 
with the initial condition $u_0(x)=\delta_0$.  Let  $G_{R,t}$ be  defined as in \eqref{GR}.  Fix  $\gamma >\frac{19}{2}$.
Then,  there exists an $R_0 \ge 1$ such that for all $R\ge R_0$
\begin{align*}
   \sup_{x\in \R} |f_{G_{R,t}}(x)-\phi(x)| \leq \frac{C_t(\log R)^\gamma}{\sqrt{ R}},
   \end{align*}
 where $f_{G_{R,t}}$ and  $\phi$ are the densities of $G_{R,t}$ and $N(0,1)$, respectively.
\end{theorem}

The organization of the paper is as follows. In Section \ref{sec2} we introduce some preliminaries on the definition of mild solution and we recall some basic elements of Malliavin calculus. In Section \ref{sec3} we establish the basic upper bound for the uniform distance between   the density of a given random variable $F$, which is a functional of an underlying isonormal Gaussian process, and the standard Gaussian density  (see Theorem \ref{densityapprox2}). In this estimate we assume that the random variable $F$ has a representation as a divergence, that is, $F=\delta(v)$.  A basic assumption  is the existence of negative moments of the projection $\langle DF, v \rangle_{L^2(\R_+\times \R)}$.  Theorem  6.2
  in \cite{HuLuNu} can be considered as a particular case of  Theorem \ref{densityapprox2} when $v= -D L^{-1} F$, where $L$ is the generator of the Orsntein-Uhlenbeck semigroup.   Section \ref{sec4} is devoted to the proof of Theorem  \ref{densityshe} and Section \ref{sec5} provides the proof of Theorem \ref{densitypam}.   
  Finally, the Appendix contains some technical lemmas used along the paper.

We remark that  a fundamental ingredient in the proof of the upper bounds \eqref{e82} and \eqref{e83}  is the  estimate of the $p$-norm of the Malliavin derivative of the solution.  More precisely, the upper bounds
\[
 \| D_{s,y} u(t,x) \| _p \le C_{T,p}  p_{t-s} (x-y)
\]
and
\[
  \| D_{s,y} u(t,x) \| _p \le C_{T,p} p_{t-s} (x-y) p_s(y)
\]
for any  $ p\ge 2$ and $0<s <t<T$, $x,y\in \R$  in Case 1 and Case 2, respectively, play a basic role in proving \eqref{e82} and \eqref{e83}.  However, when dealing with estimates for the uniform distance of the densities, in view of
inequality \eqref{e85} in
 Theorem \ref{densityapprox2}, we need  similar estimates but for the  $p$-norm of the  second Malliavin derivative. If $\sigma(x)=x$ an upper bound for 
$ \| D_{r,z} D_{s,y} u(t,x) \| _p$  has been obtained by  Chen, Khoshnevisan, Nualart and Pu \cite[Proposition 3.4]{ChKhNuPu1} when $u_0=1$ and by  Kuzgun and Nualart in \cite{KuNu}  for a general initial condition satisfying  \eqref{int}. In Case 1, for a general function $\sigma$ satisfying Hypothesis {\bf (H1)}, this estimate is established in  Proposition \ref{prop1} (see inequality \eqref{estDDu}). This inequality solves a standing open problem in this topic and it has its own interest.
 
 Along the paper   $C_t$ and $C_{t,p}$  will denote  generic  nonnegative and finite constants that depend on $t$  and $(t,p)$, respectively and might depend also on $\sigma$.

\section{Preliminaries}  \label{sec2}

We first introduce the white noise on $\R_+\times \R$.
Let $\mathcal{H}=L^2( \R_+\times \R)$ and denote by $\mathcal{B}_b\left(\R_+\times \R\right)$  the collection of Borel sets $A\in \R_+\times \R$ with finite Lebesgue measure, denoted by $m(A)$.  Consider a family of centered Gaussian random variables $\left\{ W(A): A \in \mathcal{B}_b\left(\R_+\times \R\right) \right\}$, defined on a complete probability space $(\Omega,  \mathcal{F},\P)$,  with covariance 
\[
\ex{W(A)W(B)}=m(A\cap B) \text{  for all }A,B \in \mathcal{B}_b(\R_+\times \R).
\]
The Wiener integrals
\[
W(h):=\int_{\R_+\times\R}h(s,y)W(ds,dy), \quad h\in\mathcal{H},
\]
define an {\it isonormal Gaussian process} on the Hilbert space $\mathcal{H}$.

For $t>0$, let $\mathcal{F}_t$ denote the $\sigma$-field generated by $\left\{W([0,s]\times A): 0\leq s\leq t,\ A\in \mathcal{B}_b(\R)\right\}$ and the $P$-null sets. 
A random field $X=\left\{X(s,y): (s,y)\in \R_+\times \R\right\}$ is said to be {\it adapted} if  $X(s,y)$ is $\mathcal{F}_s$-measurable for each $(s,y)\in \R_+\times \R$.
For any adapted  random field $X$, that is also jointly measurable and square integrable, i.e.
\[
\int_{\R_+}\int_{\R} \ex{X^2(s,y)}dyds <\infty,
\]
the stochastic integral \begin{align*}
\int_{\R_+\times \R} X(s,y)W(ds,dy) 
\end{align*}
is well defined (see \cite{Wa}). 

The following proposition   (see \cite{ChDa,Wa}) ensures the existence and uniqueness of a {\it mild solution} to equation \eqref{SHE} in Case 1 and Case 2.

\begin{proposition}
In Case 1, there exists a unique measurable and  adapted random field $u=\left\{u(t,x):(t,x)\in  \R_+\times \R\right\}$ such that  for all $T>0$ and $p\ge 2$
\begin{equation}  \label{estu1}
 \sup_{(t,x) \in [0,T] \times\R }\ex{ |u(t,x) |^p}  =C_{T,p} \, ,
 \end{equation}
 and  for all  $t\ge 0$ and $x\in \R$
\begin{align}\label{mild}
u(t,x)=1 +\int_{[0,t]\times \R} p_{t-s}(x-y)\sigma(u(s,y))W(ds,dy).
\end{align}
In Case 2, there exists a unique  measurable and  adapted random field   \break  $u=\left\{u(t,x):    (t,x)\in  (0, \infty) \times \R\right\}$ such that  for all $T>0$,  $t\in (0,T]$, $x\in \R$ and $p\ge 2$
\begin{equation}  \label{estu2}
  \ex{ |u(t,x) |^p}  \le C_{T,p} p_t(x)\, ,
 \end{equation}
 and  for all  $t>0$ and $x\in \R$
\begin{align}\label{mild2}
u(t,x)=p_t(x) +\int_{[0,t]\times \R} p_{t-s}(x-y)u(s,y)W(ds,dy).
\end{align} 
\end{proposition}

\subsection{Malliavin Calculus}

We recall some basic facts from Malliavin calculus.  For a broader exposition,  we refer to \cite{Nu}.
We recall that  $\mathcal{H}=L^2(\R_+\times \R)$.
Let $\mathcal{S}$ denote the collection of smooth and  cylindrical random variables of the form
\begin{equation} \label{e90}
 F=f\left(W(h_1),\dots, W(h_n)\right), 
\end{equation}
where  $f\in C^\infty_b(\R^n)$, that is,  $f$ is infinitely differentiable and all its partial derivatives are bounded, and $ h_1,\dots, h_n \in  \mathcal{H}$.
For a random variable $F \in \mathcal{S}$ of the form \eqref{e90}, we define its Malliavin derivative by putting
\[
D_{t,x}F=\sum_{i=1}^n  \frac {\partial f}{\partial x_i} (W(h_1),\cdots , W(h_n))h_i(t,x)
\]
for $t\ge 0$ and $x\in \R$.
For any real number $p\geq 1$ and any integer $k\ge 1$,  let $\D^{k,p}$ be the closure of $S$ with respect to the norm 
\[
\|F\|_{k,p}=\left(\ex{\left|F\right|^p}+ \sum_{j=1}^k\ex{\left\|D^jF\right\|^p_{  \mathcal{H}^{\otimes j} }}\right)^{1/p},
\]
where $D^j$ denotes the $j$th iterated derivative.  If  $V$ is a real  separable Hilbert space, similarly we can introduce the corresponding Sobolev spaces 
$\D^{k,p}(V)$ of $V$-valued random variables.

The divergence operator $\delta$ is defined as the adjoint of the derivative operator $D$ as an unbounded operator from $L^1(\Omega)$ to $L^1(\Omega; \mathcal{H})$. Namely, an element $v \in L^1(\Omega;  \mathcal{H})$ belongs to the domain of $\delta$, denoted by $\dom\, \delta$, if there exists an integrable random variable $\delta(v)$ verifying
\[
\ex{F\delta(v)} = \ex{\langle DF, v \rangle_{\mathcal{H}}} \, ,
\]
for any $F\in \mathcal{S}$.  A crucial property of divergence is the fact that 
any adapted and square integrable process $v$ belongs to $ \dom\, \delta $ and $\delta(v)$  coincides with the  stochastic integral: 
\[
    \delta(v)=\int_{\R_+\times \R}v(s,y)W(ds,dy). 
\]
The following lemma is a factorization property of the divergence operator, obtained in this  generality in   \cite[Lemma 1]{CaFeNu}.

\begin{lemma}\label{domdelta} Fix $p,p'>1$ with $1/p+1/p'=1$. Let $F\in \D^{1,p'}$, $v\in {\rm Dom} \,  \delta$, be such that $v \in L^{p}(\Omega; \mathcal{H}))$ and $\delta(v) \in L^{p}(\Omega)$. Then $Fv\in {\rm Dom} \, \delta$, and 
\[
\delta(Fv)=F\delta(v)- \langle DF, v \rangle_{\mathcal{H}}.
\]
    \end{lemma}
    Because $\delta $ is a continuous linear operator from $D^{1, p}(\H)$ to $L^{p}(\Omega)$, Lemma \eqref{domdelta} holds true provided $F\in \D^{1,p'}$ and
    $v\in  \D^{1, p}(\H)$.

  The operators $D$ and $\delta$ satisfy the following commutation relation 
\begin{align}
\label{commute} D_{s,y}(\delta(v))=v(s,y)+\delta(D_{s,y} v), \quad
\end{align}
for almost all $(s,y) \in \R_+\times \R$, provided
   $v\in \D^{1,2}(\Omega; \mathcal{H})$  is such that for almost all $(s,y) \in \R_+\times \R$,
   $D_{s,y} v$ belongs to the domain of the divergence in $L^2$
   and $\ex{   \int_{\R_+\times\R}  | \delta(D_{s,y} v) | ^2 dsdy } <\infty$ (see   \cite[Proposition 1.3.2]{Nu}).

\subsection{Notation}
Recall that $p_t(x):=\frac{1}{\sqrt{2\pi t}}e^{-x^2/2t}$ for $t>0$ and $x\in \R$. We will make use of the following identity. For $0<s<t$, $a,b \in \R$,
 \begin{align}\label{identity}
  p_{t-s}(a)p_{s}(b)=p_{t}(a+b)p_{\frac{s(t-s)}{t}}(b-\frac{s}{t}(a+b)).
\end{align}
For a Lipschitz function $f$,  let 
$\lip_{f}:=\sup_{x\neq y}\frac{|f(x)-f(y)|}{|x-y|}<\infty$.  We set
\begin{align}\label{sigmanotation} \Sigma_{t,x}:=\sigma(u(t,x)), \hskip 10pt \Sigma^{(1)}_{t,x}:=\sigma'(u(t,x)),\hskip 10pt \Sigma^{(2)}_{t,x}:=\sigma''(u(t,x)).
\end{align}
Set $Q_R:=[-R,R]$, and define for $0<s<t$ and $y\in \R$
\begin{align}\label{phi}
    \phi_{R,t}(s,y) := \frac 1{ \sigma_{R,t}} \int_{Q_R} {p_{t-s}(x-y)} dx,\,\,\,\,\,\,\,\,\,\,\, 
\end{align}
and
\begin{align}\label{varphi} 
    \varphi_{R,t}(s,y) := \frac 1{ \Sigma_{R,t}}\int_{Q_R} {p_{s(t-s)/t}(y-\frac{s}{t}x)} dx.
 \end{align}
For  an $\H$-valued random variable  $v: \Omega \to \mathcal{H}$ and $F \in \D^{1,1}$, we put 
\[
D_v F:= \left\langle DF,v\right\rangle_{\mathcal{H}}.
\]

\section{Existence  and estimates  of densities}  \label{sec3}
In this section, we recall the basic density formula for one-dimensional random variables, and state and prove an estimate on the uniform distance between the density of a random variable and the standard normal density,  see Theorem \ref{densityapprox2}.  The proof of this theorem follows the lines of the proof  Theorem 6.2 of \cite{HuLuNu}. For the sake of completeness, we include here some details of the proof.   These results will be later applied in the context of spatial averages of the stochastic heat equation.

The results of this section are valid in the framework of a general isonormal Gaussian process  $W=\{ W(h): h\in \H\}$ on a Hilbert space $\H$.   
The following density formula under general assumptions on the random variable has been proved in \cite[Proposition 1]{CaFeNu}.
 
\begin{proposition} 
\label{density2} Let $F\in \D^{1,1}$ and  $v \in L^{1}(\Omega; \H)$ be such that  $D_vF\neq 0$ a.s. Assume that $v/D_vF \in {\rm Dom} \, \delta$. Then the law of $F$ has a continuous and bounded density given by
\[
   f_F(x)=\E\left[\mathbf{1}_{\{F>x\}} \delta \left(\dfrac{v}{D_vF}  \right)\right].
\]
\end{proposition}

  Using Lemma \ref{domdelta},   in the context of Proposition \ref{density2}, the following constitute  sufficient conditions for $v/D_vF \in {\rm Dom}\, \delta $, for some $p, p'$ with $1/p+1/p'=1$ (see \cite[Lemma 3]{CaFeNu}):
\begin{itemize}
\item[(i)] $\left(D_vF\right)^{-1} \in \D^{1,p'}$.
\item[(ii)] $v\in \D^{1,p}( \H)$.
\end{itemize}  
In view of \cite[Lemma 4]{CaFeNu}, a sufficient condition for (i) is  $\left(D_vF\right)^{-1} \in L^{p'}(\Omega)$ and
\[
\left(D_vF\right)^{-2} \left[ \| D^2F \|_{\H \otimes \H} \| v \|_\H + \| Dv \|_{\H\otimes \H} \| DF \|_\H \right] \in L^{p'}(\Omega).
\]
Therefore, assuming that $F\in \D^{2,p}$ and $(D_vF)^{-1} \in L^q(\Omega)$, then condition (i) holds if $2/q+ 3/p =1$ for some $p>3$ and $q>2$.
In particular, we can take $q=4$ and $p=6$.

\begin{theorem}%{Theorem 6.2 }
\label{densityapprox2} 
Assume that $v \in \D^{1,6}(\Omega; \H)$ and $F=\delta(v) \in \D^{2,6}$ with $\ex{F}=0$, $\ex{F^2}=1$ and $\left( D_vF \right)^{-1} \in L^{4}(\Omega)$. Then, $v/D_vF \in {\rm Dom} \,  \delta$, $F$ admits a density $f_F(x)$ and the following inequality holds true
   \begin{align}  \nonumber
  \sup_{x\in \R}|f_F(x)-\phi(x)| \leq  & \left(\left\| {F}\right\|_{4}\left\| (D_vF)^{-1} \right\|_{4}+2\right) \left\|  1-D_vF\right\|_2\\ 
  \label{e85}
   &+\left\| (D_vF)^{-1} \right\|_{4}^2 \left\|D_v\left({D_vF}\right)\right\|_{2},
\end{align}
where $\phi(x)$ is the density of  the law $ N(0,1)$. 
\end{theorem}

\begin{proof} 
First, note that, by Proposition \ref{density2},  $F$ admits a density $f_F(x)=\E\left[\mathbf{1}_{\{F>x\}} \delta \left(\bar{v}  \right)\right]$, where $\bar{v}=v/D_vF$.
As a consequence, we can write
\begin{equation} \label{ecu1}
    \sup_{x\in \R}|f_F(x)-\phi(x)|= \sup_{x\in \R}\left|\E\left[\mathbf{1}_{\{F>x\}} \delta \left(\bar{v}  \right)\right]-\E\left[\mathbf{1}_{\{N>x\}} N\right]\right|,
\end{equation}
where $N$ denotes a $N(0,1)$ random variable.
We have
 \begin{equation}  \label{ecu2}
    \delta(\bar{v})=\delta\left(\frac{v}{D_vF}\right)=\frac{F}{D_vF}- D_v\left(\frac{1}{D_vF}\right)=\frac{F}{D_vF}+\frac{D_v\left(D_vF\right)}{\left(D_vF\right)^2}.
\end{equation} 
Indeed, the second equality follows from Lemma \ref{domdelta} together with $F=\delta(v)$, and the third one follows from the chain rule.  Then, 
substituting \eqref{ecu2} into \eqref{ecu1}, yields
\begin{align}       \nonumber
  \Phi_x:=  &\left|\E\left[\mathbf{1}_{\{F>x\}} \delta \left(\bar{v}  \right)\right]-\E\left[\mathbf{1}_{\{N>x\}} N\right]\right| \\ 
  =& \left|\ex{ \frac{\mathbf{1}_{\{F>x\}}F}{D_vF}}-\ex{ \frac{\mathbf{1}_{\{F>x\}}D_v\left(D_vF\right)}{\left(D_vF\right)^2}}-E\left[\mathbf{1}_{\{N>x\}} N\right] \right|.  \label{diff}
     \end{align} 
Adding and subtracting $\E \left[\mathbf{1}_{\{F>x\}}F\right]$ in \eqref{diff}, we get 
\begin{equation} \label{e91}
  \Phi_x \leq  \ex{\left| \frac{ (1-D_vF)F}{D_vF}\right|}+\ex{ \frac{|D_v\left(D_vF\right)|}{\left(D_vF\right)^2}}+\left|\ex{F\mathbf{1}_{\{F>x\}}- N\mathbf{1}_{\{N>x\}} }\right| .
\end{equation}
Applying H\"{o}lder's inequlity to the first term, we obtain 
\begin{equation} \label{e92}
    \ex{\left| \frac{ (1-D_vF)F}{D_vF}\right|} \leq \left\| {F}\right\|_{4}
\left\|\left(D_vF\right)^{-1} \right\|_{4} \left\|  1-D_vF\right\|_2.
\end{equation}
Meanwhile, applying H\"{o}lder's inequality to the second term, we get
\begin{equation} \label{e93}
\ex{ \frac{|D_v\left(D_vF\right)|}{\left(D_vF\right)^2}}\leq \left\|\left(D_vF\right)^{-1}\right\|_4^2\left\| D_v\left(D_vF\right)\right\|_{2}.
\end{equation}
Finally,   applying Stein's method with $h(y)=y \mathbf{1}_{\{y>x\}}$ (see \cite[Chapter 3]{NoPe}), and using the integration by parts formula $\E\left[Ff(F)\right]=\E\left[f'(F)\langle DF, v\rangle_{\H}\right]$, we have 
\begin{align} \nonumber
    \left|\E\left[F\mathbf{1}_{\{F>x\}}- N\mathbf{1}_{\{N>x\}} \right]\right| &= \E\left[  f_h'(F)-Ff_h(F)\right]\\ & = \E\left[ f_h'(F)\left(1-D_vF \right)\right]\leq   \|f_h'(F)\|_2 \|1 -D_vF\|_2,
    \label{e94}
\end{align}
where $f_h$ is the solution to the Stein equation associated to the function $h$. The estimate in  \cite[Lemma 2.1]{HuLuNu} states that 
$ | f'_h(y)| \le |y|+1$. Therefore,
\begin{align*}
    \|f_h'(F)\|_2 \leq 2.
\end{align*}
Then, substituting \eqref{e92}, \eqref{e93} and \eqref{e94}  into \eqref{e91}  yields the desired estimate.

\end{proof}

\section{Proof of Theorem \ref{densityshe}}  \label{sec4}
Before proving Theorem  \ref{densityshe}, we will show two fundamental technical ingredients:  Moment estimates for the second Malliavin derivative of the solution and negative moments of the projection  of $ DF_{R,t}$ on $v_{R,t}$, where $F_{R,t} = \delta(v_{R,t})$.

 \subsection{Moment estimates of the second derivative of u}
  Let $u$ be the solution to the stochastic heat equation \eqref{SHE}  with initial condition $u_0=1$. We know (see \cite[Proposition 5.1]{NuQu} or \cite[Proposition 2.4.4]{Nu}) that for all $(t,x) \in \R_+\times \R$, $u(t,x) \in  \cap_{p\ge 2} \D^{1,p}$ and the Malliavin derivative  $D_{s,y} u(t,x)$  solves the following linear stochastic integral equation:
  \begin{equation}   
\label{Du} D_{s,y}u(t,x)=p_{t-s}(x-y)\sigma(u(s,y))+ \int_{[s,t]\times\R} p_{t-\tau}(x-\xi)\sigma'(u(\tau,\xi))D_{s,y}u(\tau,\xi)W\left(d\tau,d\xi\right)
\end{equation}
for almost  all $s\in [0 ,t]$ and $y\in \R$.   Moreover, the following estimate holds true:
\begin{equation} \label{estDu}
\| D_{s,y} u(t,x) \| _p \le C_{T,p}  p_{t-s} (x-y)
\end{equation}
for all $0  \le s<t \le T$ and $x,y \in \R$.
  
The following proposition provides  the corresponding bound for the moments of the second derivative.

\begin{proposition}   \label{prop1}  Let $u$ be the solution to the stochastic heat equation \eqref{SHE}  with initial condition $u_0=1$, and assume hypothesis {\bf (H1)}.
 Fix $(t,x)\in \R_+ \times \R$. Then $u(t,x) \in  \cap_{p\ge 2} \D^{2,p}$ and for almost all  $0<r<s<t$, $y,z\in \R$, the second derivative $D_{r,z}D_{s,y}u(t,x)$ satisfies the following linear stochastic differential equation: 
 \begin{align} 
    D_{r,z}D_{s,y}u(t,x)& =p_{t-s}(x-y)\sigma'(u(s,y))D_{r,z}u(s,y)   \nonumber   \\ 
    \nonumber  &\quad  +\int_{[s,t]\times\R} p_{t-\tau}(x-\xi)\sigma''(u(\tau,\xi))D_{r,z}u(\tau,\xi)D_{s,y}u(\tau,\xi) W(d\tau,d\xi)  \\  
    &\quad +\int_{[s,t]\times\R} p_{t-\tau}(x-\xi)\sigma'(u(\tau,\xi))D_{r,z} D_{s,y}u(\tau,\xi)W(d\tau,d\xi). \label{DDu}
\end{align}
Moreover,  for all $0\le r< s <t\le T$ and $x,y,z\in \R$, we have
 \begin{align}\label{estDDu}
    \left\| D_{r,z}D_{s,y}u(t,x) \right\|_p \leq C_{T,p}   \Phi_{r,z,s,y}(t,x),
\end{align}
where  $C_{T,p}$ is a constant that depends  on $T$, $p$ and $\sigma$ and
\begin{align}
\label{Phi}\Phi_{r,z,s,y}(t,x) := &p_{t-s}(x-y)\\\nonumber  
&\times \left(  p_{s-r}(y-z) + \frac {p_{t-r} (z-y)+p_{t-r} (z-x)+  \mathbf{1}_{\{|y-x| > |z-y|\}} } { (s-r)^{1/4}}  \right).
 \end{align}
\end{proposition}

\begin{proof}
We will make use of the Picard iteration scheme.
For any  $(t,x)\in  \R_+\times\R$ we put $u_0(t,x)= 1$, and for $n\in\N$ we inductively define
\[
     u_{n+1}(t,x)=1+\int_{[0,t]\times\R} p_{t-\tau}(x-\xi)\sigma(u^n(\tau,\xi))W(d\tau ,d\xi).
\]
Then, for any $p\ge 2$,  there exists a constant $c_{T,p}$ such that for all $(t,x)\in [0,T]\times \R$
\begin{align}
  \label{estun} \sup_{n \in \N} \|u_n(t,x)\|_p   \leq c_{T,p}.
\end{align} 
This result is proved in  \cite[Theorem 2.4.3]{Nu} for the case of  the stochastic heat equation on $[0,1]$ with Dirichlet boundary conditions and  the proof works similarly for the equation on $\R$.

We apply the properties of the divergence operator, namely using \eqref{commute}, to get that for almost all $(s,y)\in [0,t]\times\R$ and $x\in \R$,
\begin{align}\label{Dun}
    D_{s,y}u_{n+1}(t,x)=p_{t-s}(x-y)\Sigma_{n,s,y} +\int_{[s,t]\times \R} p_{t-\tau}(x-\xi)\sigmaone__{n,\tau,\xi}D_{s,y}u_n(\tau,\xi)W(d\tau,d\xi),
\end{align}
and  for almost all $s >t$,  $D_{s,y}u_{n+1}(t,x)=0$, where we made  use of the notation \eqref{sigmanotation} with
\begin{align*}
 \Sigma_{n,s,y}:=\sigma(u_n(s,y)) \quad \text{and}  \quad  \Sigma^{(1)}_{n,\tau,\xi}:=\sigma'(u_n(\tau,\xi)).
\end{align*}
 It has also been proven in  \cite[Lemma A.1]{HuNuVi} that there is a constant $c_{T,p}$, depending on $T$ and $p$, such that for almost all $(s,y)\in [0,t] \times \R$ and  for all $(t,x)\in  [0,T] \times \R$,
\begin{align}\label{estDun}
    \sup_{n \in \N} \|D_{s,y}u_n(t,x)\|_p \leq c_{T,p} p_{t-s}(x-y).
\end{align} 
Once again using \eqref{commute} and \eqref{Dun} together with the Leibniz rule for derivatives, we have, for almost every $r,z$ such that $0<r<s<t$  and $z\in \R$,
\begin{align}  \nonumber
D_{r,z} D_{s,y} &u_{n+1}(t,x) =  p_{t-s}(x-y) \sigmaone__{n,s,y}D_{r,z} u_{n}(s, y) \\ \nonumber
&+ \int_{[s,t]\times\R} p_{t-\tau}(x-\xi) \sigmatwo__{n,\tau,\xi} D_{r,z} u_{n}(\tau,\xi)  D_{s,y} u_{n}(\tau,\xi)  W(d\tau ,d\xi) \\ 
&+ \int_{[s,t]\times\R}  p_{t-\tau}(x-\xi)\sigmaone__{n,\tau,\xi} D_{r,z} D_{s,y} u_{n}(\tau,\xi)  W(d\tau,d\xi ), \label{DDun}
\end{align}
where   $\Sigma^{(2)}_{n,\tau,\xi}:=\sigma''(u_n(\tau,\xi))$.
Applying Burkholder-Davis-Gundy inequality in \eqref{DDun}, the estimate \eqref{estDun}, Hypothesis {\bf (H1)} and the moment estimates \eqref{estun},
 for any $p\ge 2$ we have for all $(t,x) \in [0,T] \times \R$,
\begin{align}   \nonumber
\left\| D_{r,z} D_{s,y} u_{n+1}(t,x) \right\|^2_p&  \leq  C_{T,p} p^2_{t-s}(x-y)  p^2_{s-r}(y-z) \\
&+ \nonumber   C_{T,p} \int_{s}^t \int_{\R} p^2_{t-\tau}(x-\xi)  p^2_{\tau-r} (\xi-z) p^2_{\tau-s} (\xi-y)  d\xi d\tau\ \\   
&+ C_{T,p} \int_{s}^t \int_{\R} p^2_{t-\tau}(x-\xi) \| D_{r,z} D_{s,y} u_{n}(\tau, \xi)  \|^2_pd \xi d\tau, \label{burk}
\end{align}
for some constant $C_{T,p}>0$ which depends on $T$, $p$ and $\sigma$.
Let
 $J$ be the measure on $[s,t] \times \R$ defined by
\begin{align*}
J(d\tau, d \xi )  &:=       p^2_{ {\tau-r}}(\xi-z)\delta_{s,y} (d \tau,d\xi )+p_{\tau-r}^2 (\xi-z) p^2_{\tau-s} (\xi-y)   d\tau d\xi .
\end{align*}
Then, we can put the first two summands in \eqref{burk} together and rewrite this inequality as follows:
\begin{align*}
\left\| D_{r,z} D_{s,y} u_{n+1}(t,x) \right\|^2_p  & \leq C_{T,p} \int_{ [s,t] \times \R}    p^2_{t-\tau}(x-\xi) J(d \tau,d\xi)\\
& \quad +C_{T,p}  \int_{ [s,t] \times \R}  p^2_{t-\tau}(x-\xi) \left\| D_{r,z} D_{s,y} u_{n}(\tau,\xi)  \right\|^2_p  d\tau d\xi.
\end{align*}
After one iteration, this leads to
\begin{align*}
& \left\| D_{r,z} D_{s,y} u_{n+1}(t,x) \right\|^2_p  \leq   C_{T,p}  \int_{s}^t \int_{\R} p^2_{t-s_1}(x-y_1)  J(ds_1,dy_1) \\
&  +C_{T,p} ^2 \int_{s}^t\int_{\R}\int_{s}^{s_1} \int_{\R}   p^2_{t-s_1}(x-y_1)     p^2_{s_1-s_2}(y_1-y_2)  J(ds_2,dy_2) dy_1ds_1 \\
&  +C_{T,p} ^2 \int_{s}^t  \int_{s} ^{s_1} \int_{\R^2}    p_{t-s_1}^2(x-y_1) p^2_{s_1-s_2} (y_1-y_2)
\left\| D_{r,z} D_{s,y} u_{n-1}(s_2,y_2)  \right\|^2_pdy_2dy_1ds_2ds_1.
\end{align*}
If we perform $n-1$ iterations, taking into account that  $ \left\| D_{r,z} D_{s,y} u_{1}(s,y)  \right\|^2_p =0$, we obtain
\begin{align*}
& \left\| D_{r,z} D_{s,y} u_{n+1}(t,x) \right\|^2_p  \leq C_{T,p}  \int_{s}^t \int_{\R} p^2_{t-s_1}(x-y_1)  J(ds_1,dy_1)\\
&\qquad  + \sum_{k=1}^{n-1}C_{T,p} ^{k+1}
\int_{s}^t\int_{\R}\int_{s}^{s_1}\int_{\R}\cdots \int_{s}^{s_k}\int_{\R} 
 p_{t-s_1}^2(x-y_1) p_{s_1-s_{2}}^2(y_1-y_{2} )\cdots 
   \\ &\qquad \qquad  \times  p_{s_k-s_{k+1}}^2(y_k-y_{k+1} )  J(ds_{k+1},dy_{k+1}) dy_k ds_k\cdots dy_1ds_{1} .
\end{align*}
For $0\le r< s < t$,  $x,y,z \in \R$,  set
\begin{equation}\label{K}
 K^2_{r,z,s,y}(t,x):= \int_{s}^t \int_{\R}p^2_{t-s_1}(x-y_1)J(ds_1,dy_1).
\end{equation}
For the sake of simplicity, we use $K^2(t,x)$ to denote $K^2_{r,z,s,y}(t,x)$. The identity $p_t^2(x)=\frac{1}{\sqrt{2\pi t}}p_{t/2}(x)$ now implies
\begin{align*}
& \| D_{r,z} D_{s,y} u_{n+1}(t,x) \|^2_p  \leq C_{T,p}  K^2(t,x)\\
&\qquad  + \sum_{k=1}^{n-1}\frac{C_{T,p}  ^{k+1}}{(2\pi)^{\frac{k+1}{2}}}
\int_{s  <s_{k+1}< \cdots<s_2< s_1  <t} ds_1 \cdots ds_{k}  \int_{\R^{k+1}} dy_1 \cdots dy_k    \\
& \qquad   \quad \times  [ (t-s_1) (s_{1} - s_{2}) \cdots (s_k -s_{k+1})] ^{-\frac  12}  \\
&  \qquad  \quad \times p_{\frac{t-s_1}2}(x-y_1)  p_{\frac{s_1-s_{2}}2}(y_1-y_{2} )\cdots  p_{\frac{s_k-s_{k+1}}2}(y_k-y_{k+1} )   J(ds_{k+1},dy_{k+1}).
\end{align*}
Integrating in the variables $y_1, \dots, y_k$ and  using the  semigroup property of the heat kernel yields
\begin{align*}
& \| D_{r,z} D_{s,y} u_{n+1}(t,x) \|^2_p  \leq  C_{T,p}  K^2(t,x) + \sum_{k=1}^{n-1}\frac{C_{T,p}  ^{k+1}}{(2\pi)^{\frac{k+1}{2}}}
\int_{s  <s_{k+1}< \cdots<s_2< s_1  <t} ds_k \cdots ds_{1}     \\
& \qquad   \quad \times  [ (t-s_1) (s_{1} - s_{2}) \cdots (s_k -s_{k+1})] ^{-\frac  12}  \int_{\R}  p_{\frac{t-s_{k+1}}2}(x-y_{k+1} )  J(ds_{k+1},dy_{k+1}) \\
& = C_{T,p}  K^2(t,x) + \sum_{k=1}^{n-1}\frac{C_{T,p}  ^{k+1}}{(2\pi)^{\frac{k+1}{2}}}
\int_{0  <r_{k}< \cdots<r_2< r_1  <1} dr_k \cdots dr_{1}     \\
& \qquad   \quad \times  [ (1-r_1) (r_{1} - r_{2}) \cdots r_k ] ^{-\frac  12}  \int_{\R}\int_{s}^t (t-\tau)^{ \frac{k}2} p_{\frac{t-\tau}{2}}(x-\xi)  J(d\tau,d\xi) 
\\
& = C_{T,p}  K^2(t,x) + \sum_{k=1}^{n-1}  \frac{\Gamma(1/2)^{k} C ^{k+1}}{ (2\pi)^{\frac{k}{2}}\Gamma(  {k}/2)}   \int_{\R}\int_{s}^t   (t-\tau)^{ \frac{k+1}2} p^2_{{t-\tau}}(x-\xi )  J(\tau,d\xi)    \\
& \leq C K^2(t,x)+ \sum_{k=1}^{n-1}  \frac{\Gamma(1/2)^{k} C_{T,p}  ^{k+1} T^{ \frac {k+1}2}} {(2\pi)^{\frac{k}{2}} \Gamma(k/2)} \int_{\R} \int_{s}^t p_{t-\tau}^2(x-\xi)J(d\tau,d\xi) \\
&\le \left(C_{T,p} + \sum_{k=1}^{\infty}  \frac{\Gamma(1/2)^{k} C_{T,p}  ^{k+1} T^{ \frac {k+1}2}} { (2\pi)^{\frac{k}{2}}\Gamma(k/2)} \right) K^2(t,x)\\
& =:  \widetilde{C}_{T,p}^2 K^2(t,x).
\end{align*} 
Using Lemma \ref{KPhi}, we arrive at the upper-bound 
\[
\sup_{n\in \N} \| D_{r,z} D_{s,y} u_{n }(t,x) \|_p   \le  \widetilde{C}_{T,p}  \Phi_{r,z,s,y}(t,x) .
\] 
As a consequence, applying Minkowski's inequality we can write
\begin{align*}
\sup_{n\in \N}   \ex{ \left\|D^2u_n(t,x)\right\|^p_{\HH \otimes \HH}} & \leq     \sup_{n\in \N} \left(    \int_{[0,t]^2}    \int_{\R^2} \|D_{r,z} D_{s,y} u_n(t,x)\|_p^2 dy dz dr ds \right)^{\frac p2} \\
    & \le 
      \widetilde{C}_{T,p} ^p  \left( 2\int_0^t \int_{0}^{s} \int_{\R^2}\Phi_{r,z,s,y}(t,x) dzdy drds  \right)^{\frac p2} <\infty.
\end{align*}
Since $u_n(t,x)$ converges in $L^p(\Omega)$ to $u(t,x)$ for all $p\ge 2$, we deduce that $u(t,x) \in \cap_{p\ge 2} \D^{2,p}$.
By the same arguments as in the proof of Theorem 6.4 in \cite{ChKhNuPu0} we deduce the estimate
\begin{equation*} 
 \| D_{r,z} D_{s,y} u(t,x) \|_p   \le  \widetilde{C}_{T,p}  \Phi_{r,z,s,y}(t,x)
\end{equation*}
for all $0 <r<s<t<T$, $z,y,x \in \R$ and $p\ge 2$.
\end{proof}
 
\subsection{Negative moments}
 Recall  that $u$ satisfies equation \eqref{mild}.
 For any fixed $t>0$,  the  random variable $F_{R,t}$ defined in \eqref{FR} is given by
    \begin{align*} 
   F_{R,t} &=\frac{1}{\sigma_{R,t}}\left(\int_{ Q_R} \int_0^t \int_{\R} p_{t-s}(x-y)\sigma(u(s,y))W\left(ds,dy\right)dx\right) \\&=
   \int_0^t\int_{\R}  \frac{1}{\sigma_{R,t}}  \left(\int_{Q_R} p_{t-s}(x-y)\sigma(u(s,y))dx\right) W\left(ds,dy\right),
\end{align*}
where we recall that $Q_R= [-R,R]$.
So,  taking into account that the It\^o-Walsh stochastic integral coincides with the divergence operator, we obtain the representation
 \[
 F_{R,t}=\delta(v_{R,t}),
 \]
where
\begin{equation} \label{v}
    v_{R,t}(s,y)= \mathbf{1}_{[0,t]} (s) \frac{1}{\sigma_{R,t}} \int_{Q_R} p_{t-s}(x-y)\sigma(u(s,y))dx.
\end{equation}

In order to apply Theorem  \ref{densityapprox2} we need estimates on the negative moments of $D_{v_{R,t}}F_{R,t}$.
 The next proposition provides the desired estimates.

\begin{proposition}\label{F^-1} 
Let $u$ be  the solution to the integral equation \eqref{mild}  and assume that $\sigma$ is Lipschitz.
Fix $p\geq 2$, $t>0$ and assume that there exists $q>5p$ such that   $\ex{\left|\sigma(u(t,0))\right|^{-2q}}<\infty$. Then,  there exists $R_0>0$ such that 
\begin{equation}
\sup_{R\ge R_0} \ex{\left| D_{v_{R,t}}F_{R,t} \right|^{-p} }<\infty.  \label{b1}
\end{equation}
\end{proposition}
 
 \begin{proof}
 Consider the Malliavin derivative of $F_{R,t}$ given by
  \begin{align*}
   & D_{r,z}F_{R,t}=\frac{1}{\sigma_{R,t}} \int_{Q_R} dx\,D_{r,z}u(t,x).
\end{align*}
  From \eqref{v} and \eqref{FR}, we can write
 \begin{align}  \nonumber
D_{v_{R,t}}F_{R,t}&= \int_0^t \int_{\R}  v_{R,t}(r,z)D_{r,z}F_{R,t} dzdr
 \\&=\frac{1}{\sigma_{R,t}^2}\int_{Q_R^2} \int_0^t \int_{\R} p_{t-r}(x_1-z)\sigma(u(r,z))D_{r,z}u(t,x_2)dzdrdx_1 dx_2  \nonumber
 \\&=\frac{1}{\sigma_{R,t}^2}\int_{Q_R^2}\int_0^t  \int_{\R}  p_{t-r}(x_1-z)\sigma^2(u(r,z))\Psi^{r,z}(t,x_2 )dzdrdx_1 dx_2,  \label{e7}
\end{align}
with the notation
\[
\Psi^{r,z}(t,x)=  \frac{D_{r,z}u(t,x)}{\sigma(u(r,z))},
\]
for any $r<t$. Notice that  $\sigma(u(r,z)) \not =0$ almost surely because   $\ex{\left|\sigma(u(r,z))\right|^{-2q}}<\infty$ due to our hypothesis and the stationarity of the process
$\{u(r,z): z\in \R\}$.

We claim that
\begin{equation} \label{e6}
\Psi^{r,z}(t,x) \ge 0.
\end{equation}
Indeed,  from equation \eqref{Du}, it follows that$\left\{\Psi^{r,z}(t,x): (t,x)\in [r,\infty)\times \R\right\}$  satisfies: 
\begin{align*}
   \Psi^{r,z}(t,x)=p_{t-r}(x-z)+\int_{[r,t]\times \R} p_{t-s}(x-y)\sigma'(u(s,y))\Psi^{r,z}(t,x)W(ds,dy).
\end{align*} 
That means,$\Psi^{r,z}(t,x)$ solves the heat equation
 \begin{align*}
  \frac{\partial \Psi^{r,z}}{\partial t}= \frac{1}{2}  \frac{\partial^2 \Psi^{r,z}}{\partial x^2}+ \sigma'(u)  \Psi^{r,z} \Dot{W},  \qquad x\in \R, \,\,  t \in [r,\infty), 
\end{align*}  
with initial condition $\Psi^{r,z}(t,x) \big|_{t=r}=\delta_z(x)$ and,  in particular, $\Psi^{r,z}(t,x)$ is nonnegative.
 
 As a consequence, from  \eqref{e7} and  \eqref{e6}  it follows that  $D_{v_{R,t}}F_{R,t} \ge 0$ and we can write
\[
 D_{v_{R,t}}F_{R,t}
   \geq \frac{1}{\sigma_{R,t}^2}\int_{Q_R^2}\int_{t-\e^{\alpha}}^t\int_{\R}p_{t-r}(x_1-z)\sigma(u(r,z))D_{r,z}u(t,x_2)dzdrdx_1dx_2 ,
\]
for any $\e <t$ and $\alpha <1$.
 Set $\talpha:=t-\e^{\alpha}$. Using this estimate, we get
\[
 \P\left(  D_{v_{R,t}}F_{R,t} <\e \right) \leq   \P\left( \frac{1}{\sigma_{R,t}^2}\int_{Q_R^2}\int_{\talpha}^t\int_{\R}p_{t-r}(x_1-z)\sigma(u(r,z))D_{r,z}u(t,x_2)dzdrdx_1dx_2 < \e \right). 
\]
With the notation \eqref{phi},  using \eqref{Du} we obtain
\begin{align*}
 &\frac{1}{\sigma_{R,t}^2}\int_{Q_R^2}\int_{\talpha}^t\int_{\R}p_{t-r}(x_1-z)\sigma(u(r,z))D_{r,z}u(t,x_2)dzdrdx_1 dx_2 \\
 &=  \int_{\talpha}^t\int_{\R} \phi^2_{R,t}(r,z)\sigma^2(u(r,z))dzdr \\ 
 &+\int_{\talpha}^t \int_{\R} \left(\int_{[r,t]\times\R}\phi_{R,t}(s,y) \phi_{R,t}(r,z) \sigma'(u(s,y))D_{r,z}u(s,y)W(ds,dy)\right) \sigma(u(r,z)dzdr \\
 &=:I_1+I_2.
\end{align*}
From
\begin{align}
  \label{prob} \P\left(I_1+I_2 < \e\right)&\leq   \P\left(I_1 < 2\e\right)+\P\left(I_1+I_2 < \e, \, I_1\geq 2\e\right)\\ \nonumber
    &\leq \P\left(I_1 <2 \e\right)+\P\left( |I_2|> \e\right),  
\end{align}
 we have
\begin{align*}
& \P\left(\frac{1}{\sigma_R^2}\int_{Q_R^2}\int_{\talpha}^t\int_{\R}p_{t-r}(x_1-z)\sigma(u(r,z))D_{r,z}u(t,x_2)dzdrdx_1 dx_2< \e \right) \\
&\leq \P\left( I_1< 2\e \right) +\P\left(\left|  I_2 \right|> \e \right) .
\end{align*}
We shall next estimate these probabilities, starting with the first one: 
\begin{align*}
 \P\left( I_1< 2\e \right) &=  \P\left( \int_{\talpha}^t\int_{\R}\phi_{R,t}^2(r,z)\sigma^2(u(r,z))dzdr< \e \right) \\&= \P\left( \int_{\talpha}^t\int_{\R} \phi_{R,t}^2(r,z) \bigg(\sigma(u(r,z))-\sigma(u(t,z))+\sigma(u(t,z))\bigg)^2 dzdr < 2\e \right).
\end{align*}
Using the inequality  $(a+b)^2 \geq a^2/2-b^2$ for $a,b \in \R$,  and an   estimate similar to \eqref{prob}, we get 
\begin{align}   \nonumber
& \P\left( \int_{\talpha}^t\int_{\R} \phi_{R,t}^2(r,z) \left(\sigma(u(r,z))-\sigma(u(t,z)+\sigma(u(t,z)) \right)^2 dzdr < 2\e \right) \\
 \nonumber &  \leq \P\left( \int_{\talpha}^t \int_{\R}\phi_{R,t}^2(r,z) \left(\sigma(u(t,z))\right)^2 dzdr < 6\e \right) \\ 
 & \quad +\P\left(\int_{\talpha}^t \int_{\R}\phi_{R,t}^2(r,z) \left(\sigma(u(r,z))-\sigma(u(t,z))\right)^2 dzdr > \e \right)  \nonumber \\
&=: K_1+ K_2.
\label{K1K2} 
\end{align}
For the term $K_1$ in \eqref{K1K2}, by Chebyshev's inequality, for $q>5p$ we obtain
\begin{align}  \nonumber
 K_1 & =\P\left(   \left[ \int_{\talpha} ^t \int_\R    \phi_{R,t}^2(r,z) \sigma^2(u(t,z)) dz dr \right] ^{-1} > \frac 1{6\e} \right) \\
 &  \le (6\e)^q  \ex{ \left( \int_{\talpha} ^t \int_\R    \phi_{R,t}^2(r,z) \sigma^2(u(t,z)) dz dr \right) ^{-q}}.  \label{e8}
\end{align}
Set
\[
m(\e, R):=  \int_{\talpha} ^t \int_\R    \phi_{R,t}^2(r,z)  dz dr .
\]
Then, taking into account that the function $x\to x^{-q}$ is convex   and applying Jensen's inequality, we can write
\begin{align}  \nonumber
&  \ex{  \left(\int_{\talpha} ^t \int_\R    \phi_{R,t}^2(r,z) \sigma^2(u(t,z)) dz dr  \right) ^{-q}} \\  \nonumber
&= m(\e, R)^{-q}  \ex{ \left(  \frac 1{ m(\e,R)}  \int_{\talpha} ^t \int_\R    \phi_{R,t}^2(r,z) \sigma^2(u(t,z)) dz dr  \right)^{-q}} \\  \label{jensen}
& \le   {m(\e, R)^{-q-1} }    \int_{\talpha} ^t \int_\R    \phi_{R,t}^2(r,z) 
\ex{ |\sigma (u(t,z)) | ^{-2q} } dr dz.
\end{align}
Since the solution is stationary in space, the factor  $C_t:=\ex{ |\sigma (u(t,z)) | ^{-2q} } $ does not depend on $z$ and we assume  it is finite. 
Therefore, from \eqref{e8} and \eqref{jensen}, we get
\begin{equation} \label{e9}
K_1 \le  C_t(6\e)^q  m(\e, R)^{-q}
\end{equation}
for some constant $C_t>0$.
Moreover,
\begin{align*}
m(\e, R) &=\frac {1}{\sigma_{R,t}^2}
  \int_0 ^{\e^\alpha}    \int_{-R} ^R \int_{-R} ^Rp_{2s} (x_1-x_2) dx_1d x_2  ds \\
&\ge \frac { \sqrt{2} R}{\sigma_{R,t}^2}    \int_0 ^{\e^\alpha}    \int_{-R/\sqrt{2}} ^{R/\sqrt{2}}  p_{2s} (y) dy ds .
\end{align*}
Then, assuming $\e \le 1$ and $R\ge R_0$, we obtain
\begin{equation} \label{e10}
m(\e, R)  \ge   \frac { \sqrt{2} R}{\sigma_{R,t}^2}  \int_0 ^{\e^\alpha}    \int_{-1/\sqrt{2}} ^{1/\sqrt{2}}  p_{2} (y) dy ds  \ge C_t \e^\alpha,
\end{equation}
 where in the last inequality we have used  Lemma \ref{variances} part (a).
Hence, from \eqref{e9} and \eqref{e10},  we have
\begin{equation} \label{a2}
K_1 \leq C_t \e^{q (1-\alpha)}.
\end{equation}
In order to estimate the term $K_2$ in \eqref{K1K2}, we use Chebyschev's inequality followed by Minkowski's inequality,   as follows: 
\begin{align} \nonumber
   K_2  &\leq \e^{-q} \ex{  \left( \int_{\talpha}^t\int_{\R} \phi_{R,t}^2(r,z) (\sigma(u(r,z))-\sigma(u(t,z)))^2 dzdr \right)  ^{q}} \\ 
   &\leq \e^{-q} \left(\int_{\talpha}^t\int_{\R}  \phi_{R,t}^2(r,z) \left(\ex{\left|\sigma(u(r,z))-\sigma(u(t,z))\right|^{2q}} \right)^{1/q}dzdr \right)^q. \label{eq7}
   \end{align}
    The Lipschitz continuity of $\sigma$ and the $1/4$-H\"older continuity of the solution $u(t,x)$ in $L^{2q}(\Omega)$ allow us to write 
    for any $r\in [t_\alpha, t]$
\begin{align}  \nonumber
  \| \sigma(u(r,z))-\sigma(u(t,z)) \|_{2q} &\leq \lip_{\sigma}   \| u(r,z)-u(t,z) \|_{2q} \\& \leq C_t\lip_{\sigma} |t-r|^{1/4} \leq C_t\lip_{\sigma} \e^{\alpha/4}. \label{eq8}
\end{align}
On the other hand,
 from \eqref{e10} we have, for $R\ge R_0$, 
\begin{align} \nonumber
    &\int_{\talpha}^t\int_{\R} \phi_{R,t}^2(r,z)dzdr=\frac{1}{\sigma_{R,t}^2}\int_{0}^{\e^{\alpha}}\int_{Q^2_R}\int_{\R} p_{r}(x_1-z)p_{r}(z-x_2)dz dx_1dx_2dr \\ \label{eq9}
    &\leq\frac{1}{\sigma_{R,t}^2}\int_{0}^{\e^{\alpha}}\int_{Q_R^2}  \int_{\R} p_{2r}(x_1-x_2)  dx_1dx_2dr \leq \frac{2R}{\sigma_{R,t}^2}\e^{\alpha}\leq C_t \e^{\alpha}.
\end{align}
Substituting \eqref{eq8} and \eqref{eq9} into \eqref{eq7}, yields
\begin{equation} \label{a3}
  K_2 \leq  C_t \e^{(\frac{3\alpha}{2}-1)q}.
\end{equation}
We are left to estimate the following probability:
\begin{align*}
    K_3:=&\P\left(\left|I_2\right|> \e \right) .
\end{align*}
Using Fubini's theorem and Chebyschev's inequality, we have 
\[
   K_3  \leq \frac{1}{\e^q}\ex{\left|\int_{[\talpha,t]\times\R}\int_{\R}\int_{\talpha}^s  \phi_{R,t}(r,z)\phi_{R,t}(s,y)\sigmaone__{s,y}D_{r,z}u(s,y)\Sigma_{r,z}drdzW(ds,dy)\right|^q}.
\]
Then, applying  Burkholder-Davis-Gundy inequality, followed by Minkowski's inequality,  we get
\begin{align}   \nonumber
  K_3  & \leq \frac{C_q}{\e^q}\ex{\left|\int_{\talpha}^t\int_{\R}\left(\int_{\R}\int_{\talpha}^s  \phi_{R,t}(r,z)\phi_{R,t}(s,y)\sigmaone__{s,y}D_{r,z}u(s,y)\Sigma_{r,z}drdz\right)^2  dsdy \right|^{\frac q2}}  
    \\ \nonumber &=\frac{C_q}{\e^q} \E \Bigg[ \Bigg|\int_{\talpha}^t\int_{\talpha}^s \int_{\talpha}^s\int_{\R^3} \phi_{R,t}(r_1,z_1)\phi_{R,t}(r_2,z_2)\phi^2_{R,t}(s,y) \\  \nonumber
    & \qquad  \times X_{r_1,z_1,r_2,z_2}(s,y) dz_1 dz_2 dy dr_1 dr_2  ds\Bigg|^{\frac q2} \Bigg]
    \\ \nonumber
     &    \leq \frac{C_q}{\e^q}\Bigg( \int_{\talpha}^t\int_{\talpha}^s \int_{\talpha}^s\int_{\R^3}\phi_{R,t}(r_1,z_1)\phi_{R,t}(r_2,z_2)\phi^2_{R,t}(s,y) \\
    & \qquad \times \left\|X_{r_1,z_1,r_2, z_2}(s,y)\right\|_{q/2}dz_1 dz_2  dy dr_1dr_2 ds\Bigg)^{\frac q2}, \label{BDG}
\end{align}
where 
\begin{align*}
  X_{r_1,z_1,r_2,z_2} (s,y):=&\left(\sigmaone__{s,y}\right)^2D_{r_1,z_1}u(s,y)D_{r_2, z_2 }u(s,y)\Sigma_{r_1,z_1}\Sigma_{r_2,z_2}.
\end{align*}
Using H\"older's inequality,   the Lipschitz property of $\sigma$,  the estimate \eqref{estDu}  and the fact that $\sup_{(r,z) \in [0,t]\times\R}\|\sigma(u(r,z))\|_{p} <\infty$ for all
$p\ge 2$,   we have
 \[
  \|   X_{r_1,z_1,r_2,z_2} (s,y) \|_{q/2} \leq C_t p_{s-r_1}(y-z_1) p_{s-r_2}(y-z_2).
\]
Plugging this bound  in the estimate \eqref{BDG}, we see that
\begin{align}  \nonumber
    K_3 & \leq   \frac{C_t}{\e^q} \Big( \int_{\talpha}^t\int_{\talpha}^s \int_{\talpha}^s\int_{\R^3}\phi_{R,t}(r_1,z_1)\phi_{R,t}(r_2,z_2)\phi^2_{R,t}(s,y) \\
    & \qquad  \times p_{s-r_1}(y-z_1) p_{s-r_2}(y-z_2)
    dz_1 dz_2 dy dr_1 dr_2  ds \Big)^{\frac q2} . \label{a1}
\end{align}
Integrating in $z_1$ and $z_2$, and using the semigroup property, we have for $t_{\alpha}<s<t$ and for $R\ge R_0$,
\begin{align*}
\int_{\R^3}& \left(\prod_{i=1,2}\phi_{R,t}(r_i,z_i)\phi_{R,t}(s,y) p_{s-r_i}(y-z_i)\right)dz_1dz_2dy
\\ &=\frac{1}{\sigma_{R,t}^2}\int_{\R}\int_{Q_R^2}  \prod_{i=1,2}\phi_{R,t}(s,y) p_{t+s-2r_i}(y-x_i)dx_1dx_2 dy
\\ &\leq \frac{1}{\sigma_{R,t}^2}\int_{\R}\phi^2_{R,t}(s,y)\left( \prod_{i=1,2} \int_{\R} p_{t+s-2r_i}(y-x)dx\right)dy 
\\&= \frac{1}{\sigma_{R,t}^2}\int_{\R}\phi^2_{R,t}(s,y)dy \leq  \frac{C_t}{\sigma_{R,t}^2} \leq C_t, 
\end{align*}
where we use Lemma \ref{phivarphi}  part (a) of the Appendix.  Now,  plugging this estimate in \eqref{a1},  we get 
\begin{equation} \label{a4}
    K_3\leq   \frac{C_T}{\e^q} \left( \int_{\talpha}^t\int_{\talpha}^s \int_{\talpha}^s  dr_1dr_2ds\right)^{\frac q2}= C_{T}\e^{(\frac{3}{2}\alpha -1)q}.
\end{equation}
Now,  choosing  $\alpha=4/5$, we get  from  \eqref{a2}, \eqref{a3} and \eqref{a4},
\begin{align*}
     \sup_{R\ge R_0} \P\left(D_{v_{R,t}}F_{R,t}) < \e \right) \leq C_T \e^{q/5}.
\end{align*}
Finally,  using this estimate we get \begin{align*}
\sup_{R\ge R_0} \ex{\left( D_{v_{R,t}}F_{R,t} \right)^{-p} } &= \sup_{R\ge R_0}p\int_0^{\infty} \e^{-p-1} \P\left(D_{v_{R,t}}F_{R,t} <\e\right)d\e  \\
 &\leq 1+ \sup_{R\ge R_0}p \int_0^1 \e^{-p-1}\P\left(D_{v_{R,t}}F_{R,t} <\e\right)d\e  
 \\&\leq 1+C_T p \int_0^1 \e^{-p-1+q/5} d\e  <\infty
\end{align*}
for $q>5p$, which completes our proof.

\end{proof}

%%%%%%%%%%%%%%%%%%%%%%%%%%%%%%%%%%%%%%%%%%%%%%%%%%%%%%%%%%%%%%%%%%%%%%%%%%%%%%%%%%%%%%%%%%%%%%%%%%%%%%%%%%%%%%%

Now, we are ready to prove Theorem \ref{densityshe}.

\begin{proof}[Proof of Theorem \ref{densityshe}]
We will apply Theorem \ref{densityapprox2} to the random variable  $F_{R,t} = \delta (v_{R,t})$. 
Fix  $t>0$. From the proof of Theorem 1.1 in \cite{HuNuVi},  we have 
\begin{equation} \label{e70}
\left\|1-D_{v_{R,t}}F_{R,t}\right\|_{2}\leq \frac{C_t}{\sqrt{R}}.
\end{equation}
We are only left to estimate the term $\left\| D_{v_{R,t}}\left({D_{v_{R,t}}F_{R,t}}\right)\right\|_{2}$.  Recall that

\begin{align*}
   D_{v_{R,t}}F_{R,t}&= \frac{1}{\sigma_{R,t}}\int_0^t\int_{\R}\int_{Q_R} \phi_{R,t}(s,y)\sigma(u(s,y)) D_{s,y} u(t,x)  dx dy ds.
\end{align*}
Taking  the Malliavin derivative,  we get 
\begin{align*}
D_{r,z} \left(D_{v_{R,t}}F_{R,t}\right)=\frac{1}{\sigma_{R,t} }\int_{r}^t\int_{\R}\int_{Q_R} \phi_{R,t}(s,y)\sigma'(u(s,y)) D_{r,z} u(s,y) D_{s,y} u(t,x)dxdy ds
\\ +\frac{1}{\sigma_{R,t}}\int_0^t\int_{\R}\int_{Q_R} \phi_{R,t}(s,y)\sigma(u(s,y)) D_{r,z} D_{s,y} u(t,x)dxdyds,
\end{align*}
and,  using the notation \eqref{sigmanotation}, we get
\begin{align*}
D_{v_{R,t}}& \left(D_{v_{R,t}}F_{R,t}\right)\\=&\frac{1}{\sigma_{R,t}}\int_{0}^t\int_{r}^t\int_{\R^2}\int_{Q_R} \phi_{R,t}(r,z)\phi_{R,t}(s,y)\Sigma_{r,z}\sigmaone__{s,y} D_{r,z} u(s,y) D_{s,y} u(t,x)dxdy dz dsdr
\\ &+\frac{2}{\sigma_{R,t}}\int_{0}^t\int_{r}^{t}\int_{\R^2}\int_{Q_R}\phi_{R,t}(r,z)\phi_{R,t}(s,y)\Sigma_{r,z}\Sigma_{s,y} D_{r,z} D_{s,y} u(t,x)dxdydz dsdr.
\end{align*}
Now using \eqref{Du} and \eqref{DDu} for $D_{s,y}u(t,x)$ and $D_{r,z}D_{s,y}u(t,x)$, respectively,   we have 

\begin{align*}
D_{v_{R,t}} \left(D_{v_{R,t}}F_{R,t}\right)= 
2\y^1_{R,t}+\y^2_{R,t}+2\y^3_{R,t}+2\y^4_{R,t},
\end{align*}
where: \begin{align*}
    \y^1_{R,t}&=  \int_{0}^t\int_{r}^t\int_{\R^2}dydz dsdr \phi_{R,t}^2(s,y)\phi_{R,t}(r,z) \Sigma_{r,z}\Sigma_{s,y} \sigmaone__{s,y}D_{r,z}u(s,y), \\
    & \\
    \y^2_{R,t}&= \int_{0}^t\int_{r}^t\int_{\R^2}dy dz dsdr \phi_{R,t}(s,y)\phi_{R,t}(r,z) \Sigma_{r,z} \sigmaone__{s,y}D_{r,z}u(s,y) \\ 
    & \hspace{3cm} \times \int_{[s,t]\times\R}\phi_{R,t}(\tau,\xi) \sigmaone__{\tau,\xi}D_{s,y}u(\tau,\xi) W(d\tau,d\xi),  \\
    \y^3_{R,t}&= \int_{0}^t\int_{r}^{t}\int_{\R^2}dydz ds dr \phi_{R,t}(s,y)\phi_{R,t}(r,z) \Sigma_{r,z}\Sigma_{s,y}  \\ 
    & \hspace{3cm} \times \int_{ [s,t]\times\R}\phi_{R,t}(\tau,\xi) \sigmatwo__{\tau,\xi}D_{r,z}u(\tau,\xi)D_{s,y}u(\tau,\xi) W(d\tau,d\xi),\\
    \y^4_{R,t}&=\int_{0}^t\int_r^{t}\int_{\R^2}dydz dsdr \phi_{R,t}(s,y)\phi_{R,t}(r,z) \Sigma_{r,z}\Sigma_{s,y} 
     \\ 
    & \hspace{3cm} \times \int_{[ s, t]\times\R}\phi_{R,t}(\tau,\xi) \sigmaone__{\tau,\xi}D_{r,z}D_{s,y}u(\tau,\xi) W(d\tau,d\xi).
\end{align*}

Putting together the terms  $\y^i_{R,t}$ for $i=2,3,4$, we can write
\[
D_{v_{R,t}} \left(D_{v_{R,t}}F_{R,t}\right)= 
2\y^1_{R,t}+ \y^5_{R,t},
\]
where
\[
 \y^5_{R,t} = 
  \int_0^t \int_{\R}    
  \left( \int_0^{\tau}  \int_{r} ^{\tau} \int _{\R^2}  \phi_{R,t} (s,y)  \phi_{R,t} ( r, z) Z_{r,z, s,y}(\tau ,\xi )  ds dr dy dz  \right) \phi_{R,t} (\tau,\xi) W(d\tau,d\xi),
\]
and we are using the notation
\begin{align}  \nonumber
Z_{r,z, s,y}(\tau ,\xi )  &=: \Sigma_{r,z} \Sigma^{(1)}_{s,y} \Sigma^{(1)}_{\tau,\xi} D_{r,z} u(s,y)D_{s,y} u(\tau, \xi) \\ \nonumber
& \quad + 2 \Sigma_{r,z}\Sigma_{s,y} \Sigma^{(2)}_{\tau,\xi}   D_{r,z} u(\tau,\xi)D_{s,y} u(\tau, \xi)  \\  \label{Z}
& \quad +  2\Sigma_{r,z}\Sigma_{s,y} \Sigma^{(1)}_{\tau,\xi}   D_{r,z} D_{s,y} u(\tau,\xi).
\end{align}
Therefore,
\[
\|D_{v_{R,t}} \left(D_{v_{R,t}}F_{R,t}\right)\|_2 \le 
2\|\y^1_{R,t}\|_2 + \|\y^5_{R,t}\|_2.
\]
 
 \medskip
 \noindent
{\it Estimation of}   $\left\|\y^1_{R,t}\right\|_2$: Note that using the estimates \eqref{estu1} and  \eqref{estDu} and H\"older's inequality we have, for $r < s$,
\[
    \left\|\Sigma_{r,z}\Sigma_{s,y}\sigmaone__{s,y}D_{r,z}u(s,y)\right\|_2  \leq  C_t p_{s-r}(z-y).
\]
    As a consequence,
\[
     \left\|\y^1_{R,t}\right\|_2\leq  C_t \int_0^t \int_{r}^t \int_{\R^2} \phi^2_{R,t}(s,y)\phi_{R,t}(r,z)p_{s-r}(z-y)dydz dsdr.
\]
    Integrating in $z$ and  using the semigroup property, we have
     \begin{align*}
    \int_{\R}\phi_{R,t}(r,z)p_{s-r}(z-y)dz&=\frac{1}{\sigma_{R,t}} \int_{Q_R}\int_{\R} p_{t-r}(x-z)p_{s-r}(z-y)dz dx \\ &=
    \frac{1}{\sigma_{R,t}} \int_{Q_R}p_{t+s-2r}(x-y)dx \leq \frac{1}{\sigma_{R,t}}.
    \end{align*}
    Using the above estimate, and Lemma \ref{phivarphi} part (a), Lemma \ref{variances} part (a) and  we get, for $R\ge R_0$, 
   \[
     \left\|\y^1_{R,t}\right\|_2 \leq \frac{  C_t}{\sigma_{R,t}}\int_0^t \int_{r}^t \int_{\R} \phi^2_{R,t}(s,y) dy dsdr
    \leq  \frac{ C_t  }{\sigma_{R,t}} \leq  \frac{ C_t }{\sqrt{R}}.
\]

      \medskip
 \noindent
{\it Estimation of}   $\left\|\y^5_{R,t}\right\|_2$:  Using  the Ito-Walsh isometry of the stochastic integral and Cauchy-Schwarz inequality, we obtain
 \begin{align*}
       \left\|\y^5_{R,t}\right\|_2^2   &= \int_0^t \int_{\R}    \ex{ \left( \int_0^{\tau} \int_{r} ^\tau \int _{\R^2}  \phi_{R,t} (s,y)  \phi_{R,t} ( r, z) Z_{r,z,s, y}(\tau,\xi)  ds dr dy dz  \right) ^2}  \phi^2_{R,t} (\tau,\xi) d\xi d\tau \\
  &= \int_0^t \int_{\R}      \int_{\substack{0\leq r_1 \leq  s_1\leq \tau \\0\leq r_2 \leq s_2 \leq \tau}}\int_{\R^4}  \prod_{i=1,2}dy_i dz_idr_i ds_i \phi_{R,t}(s_i,y_i)\phi_{R,t}(r_i,z_i)   \\
  & \qquad \times \|   Z_{r_i,z_i, s_i, y_i}(\tau,\xi)\| _2 \phi^2_{R,t} (\tau,\xi) d\xi d\tau .
  \end{align*}
From  the decomposition \eqref{Z}, using H\"older's inequality and the estimates \eqref{estu1}, \eqref{estDu} and \eqref{estDDu}, we can write
\begin{align*}
       \left\|\y^5_{R,t}\right\|_2^2   & \le   C_t\int_0^t \int_{\R}  d\xi d\tau  \phi^2_{R,t} (\tau,\xi)    \int_{\substack{0\leq r_1 \leq  s_1\leq \tau \\0\leq r_2 \leq s_2 \leq \tau}}\int_{\R^4}  \prod_{i=1,2}dy_i dz_idr_i ds_i \phi_{R,t}(s_i,y_i)\phi_{R,t}(r_i,z_i)   \\
   &   \qquad   \times   \left[ p_{s_i-r_i} (y_i-z_i) p_{\tau-s_i} ( \xi-y_i) + p_{\tau-r_i} (\xi-z_i) p_{\tau-s_i} ( \xi-y_i)   + \Phi_{r_i,z_i,s_i,y_i}(\tau,\xi)    \right].
       \end{align*}
The estimates $\phi_{R,t}(r_i,z_i), \phi_{R,t}(r_i,z_i) \leq \frac{1}{\sigma_{R,t}}$  imply
\begin{align*}
       \left\|\y^5_{R,t}\right\|_2^2   & \le   \frac {C_t}{ \sigma_{R,t}^2}\int_0^t \int_{\R} d\xi d\tau  \phi^2_{R,t} (\tau,\xi)    \int_{\substack{0\leq r_1 \leq  s_1\leq \tau \\0\leq r_2 \leq s_2 \leq \tau}}\int_{\R^4}  \prod_{i=1,2}dy_i dz_idr_i ds_i     \\
   &   \quad   \times   \left[ p_{s_i-r_i} (y_i-z_i) p_{\tau-s_i} ( \xi-y_i) + p_{\tau-r_i} (\xi-z_i) p_{\tau-s_i} ( \xi-y_i)   + \Phi_{r_i,z_i,s_i,y_i}(\tau,\xi)    \right].
       \end{align*}
       Integrating the variables $z_i$ and $y_i$ for $i=1,2$ and using  
 Lemma \ref{l1Phi}, we have 
\[
       \left\|\y^5_{R,t}\right\|_2^2    \le   \frac C{ \sigma_{R,t}^2}         \int_0^t \int_{\R} \phi^2_{R,t} (\tau,\xi)     \left(  \int_{0<r  <s< \tau }\left(1+(s-r)^{-1/4}\right)  dr ds \right)^2
     d\xi d\tau.
\]
    Using the above estimate, Lemma \ref{phivarphi} part (a), and Lemma \ref{variances} part (a) we finally have for $R\ge R_0$
    \begin{equation} \label{b2}
       \left\|\y^5_{R,t}\right\|_2    \le  \frac {C_t }{ \sqrt{R}}.
       \end{equation}
Finally, plugging the estimates \eqref{b1}, \eqref{e70} and \eqref{b2} into \eqref{e85} we  complete the proof.
 \end{proof}

\section{Proof of Theorem \ref{densitypam}}  \label{sec5}
 Let $u=\{ u(t,x): (t,x) \in \R_+ \times \R\}$ be the  solution to equation \eqref{mild2}. The process $u$ is no longer stationary in the space variable, but if we define
  $U$   as 
  \begin{align*}
    U(t,x):=\frac{u(t,x)}{p_t(x)}
\end{align*}
for  $(t,x) \in (0,\infty)\times \R$, then for any $t>0$, the process  $\{U(t,x): x \in \R\}$ is stationary,  see \cite{AmCoQu}.

It has  been proven in \cite{ChKhNuPu2} that $\lim_{t\downarrow 0}U(t,x)=1$ in $L^p(\Omega)$ for all $x\in \R$ and $p\geq 2$.  Moreover,  equation \eqref{mild2} can be reformulated in terms of $U$ as follows
\begin{align} \label{U}
    U(t,x)=1+\int_{[0,t]\times\R}{p_{\frac{\tau(t-\tau)}{t}}(\xi-\frac{\tau}{t}x)} U(s,y)W(d\tau,d\xi), 
\end{align} 
where we used the identity \eqref{identity}.

According to Chen, Hu and Nualart \cite[Proposition 5.1]{ChHuNu},   for any $t>0$  and any $x\in \R$,  the random variable $u(t,x)$ belongs to the Sobolev space $\mathbb{D}^{k,p}$ for any $k\ge 1$ and $p\ge 2$.  As a consequence,  for all $t>0$ and $x\in \R$, $U(t,x) \in \cap_{k\ge 1} \cap_{p\ge 2} \mathbb{D}^{k,p}$.  Furthermore, 
 for almost all $(s,y)\in (0,t)\times \R$, using \eqref{commute} and \eqref{U}, we have,
\begin{align}
   \label{DU} D_{s,y}U(t,x)=p_{\frac{s(t-s)}{t}}(y-\frac{s}{t}x)U(s,y)+\int_{[s,t]\times\R}{p_{\frac{\tau(t-\tau)}{t}}(\xi-\frac{\tau}{t}x)} D_{s,y}U(\tau, \xi)W(d\tau,d\xi),
\end{align}
and for  almost all $r\leq s\leq t $ and $y,z \in \R$,  
\begin{align}   \nonumber
    D_{r,z}D_{s,y}U(t,x)=&p_{\frac{s(t-s)}{t}}(y-\frac{s}{t}x)D_{r,z}U(s,y)\\  \label{DDU}
    &+\int_{[s,t]\times\R}{p_{\frac{\tau(t-\tau)}{t}}(\xi-\frac{\tau}{t}x)} D_{r,z} D_{s,y}U(\tau, \xi)W(d\tau,d\xi).
\end{align}
Let $G_{R,t}$ and $\Sigma_{R,t}$ be as defined in  \eqref{GR}. 
Then, for any fixed $t> 0$, $G_{R,t}=\delta(w_{R,t})$, where
 \begin{align}  \nonumber
    w_{R,t}(s,y)&= \mathbf{1}_{[0,t]} (s) \frac{1}{\Sigma_{R,t}} \int_{Q_R} {p_{\frac{s(t-s)}{t}}(y-\frac{s}{t}x)}U(s,y)dx\\ \label{w} &= \mathbf{1}_{[0,t]} (s) \varphi_{R,t}(s,y) U(s,y),
\end{align}
and $ \varphi_{R,t}(s,y)$ has been defined in \eqref{varphi}.
Finally, we also note that \begin{align*}
    D_{s,y}G_{R,t}=\frac{1}{\Sigma_{R,t}}\int_{Q_R} D_{s,y}U(t,x),
\end{align*}
and using \eqref{w}
\begin{align}\label{DwG}
    D_{w_{R,t}} G_{R,t}&=\frac{1}{\Sigma_{R,t}} \int_{0}^t\int_{\R}\int_{Q_R}\varphi_{R,t}(s,y) U(s,y) D_{s,y}U(t,x)dxdyds.
\end{align}

The following moment estimates hold for all $p\ge 2$ and for all $(t,x)\in (0,T]\times \R$ and almost all $0<r<s<t$ and $y,z\in \R$:
\begin{equation} \label{e20}
 \left\|u(t,x)\right\|_p\leq c_{T,p} p_t(x),
\end{equation}
 \begin{equation} \label{e21}
 \left\|D_{s,y}u(t,x)\right\|_p\leq c_{T,p} p_{t-s}(x-y)p_s(y),
\end{equation}
and
\begin{equation} \label{e22}
\left\| D_{r,z} D_{s,y}u(t,x)\right\|_p\leq c_{T,p} p_{t-s}(x-y)p_{s-r} (y-z)p_r(z),
\end{equation}
where $c_{T,p} $ is a constant depending only on $T$ and $p$. We refer to   Chen and Dalang
\cite[Theorem 2.4]{ChDa} for the proof of \eqref{e20}, to Chen, Khoshnevisan, Nualart and Pu \cite[Lemma 2.1]{ChKhNuPu2} for the proof of \eqref{e21} and  Kuzgun and Nualart
\cite[Corollary 1.2]{KuNu} for the proof of \eqref{e22}. Dividing by the factor $p_t(x)$ and using  the identity \eqref{identity} we derive  the corresponding estimates for the process $U(t,x)$:
\begin{equation} \label{estU}
 \sup_{(t,x)\in [0,T]\times \R}   \left\|U(t,x)\right\|_p\leq c_{T,p},
\end{equation}
 \begin{equation} \label{estDU}
\left\|D_{s,y}U(t,x)\right\|_p\leq c_{T,p} p_{ \frac {s(t-s)}t}(y-\frac st x ).
\end{equation}
and
\begin{equation} \label{estDDU}
\left\| D_{r,z} D_{s,y}u(t,x)\right\|_p\leq c_{T,p} p_{\frac {s(t-s)}t }(y-\frac st x)p_{\frac {r(s-r)} s } (z- \frac rs y).
\end{equation}

The next   proposition ensures the existence of negative moments required in the application of Theorem 
\ref{densityapprox2}.

\begin{proposition} \label{G^-1}
Fix $t \in  (0,T]$, $p\ge 2$ and $\gamma >5$.  Then,  there exist $R_0>1$ and a constant
  $c_{t,p,\gamma}$, depending on $t$, $p$ and $\gamma$,  such that 
  \[
     \left\|\left(D_{w_{R,t}}G_{R,t} \right)^{-1}\right\|_p\leq c_{t,p,\gamma}(\log R)^{\gamma}
     \]
  for all $R\ge R_0$.
      \end{proposition}

\begin{proof} Using \eqref{DwG} and \eqref{DU}, we have
 \begin{align*}
   D_{w_{R,t}} G_{R,t}&=\frac{1}{\Sigma_{R,t}} \int_{0}^t\int_{\R}\int_{Q_R}\varphi_{R,t}(s,y) U(s,y) D_{s,y}U(t,x)dxdyds\\&
   = \int_0^t \int_{\R} \varphi_{R,t}^2(s,y)U^2(s,y)dyds\\
   & \quad +
   \int_0^t\int_{\R} \varphi_{R,t}(s,y)U(s,y) \left(\int_{[s,t]\times \R} \varphi_{R,t}(\tau,\xi)D_{s,y}U(\tau,\xi)W(d\tau,d\xi)\right) dyds.
\end{align*}
Since $U$ and $DU$ are non-negative,  $D_{w_{R,t}}G_{R,t} \geq 0$ and we have 
\begin{align*}
D_{w_{R,t}}G_{R,t}
 &\geq   \int_{\talpha}^t \int_{\R} \varphi_{R,t}^2(s,y)U^2(s,y)dyds\\
 & \quad +  
   \int_{\talpha}^t\int_{\R} \varphi_{R,t}(s,y)U(s,y) \left(\int_{[s,t]\times \R} \varphi_{R,t}(\tau,\xi)D_{s,y}U(\tau,\xi)W(d\tau,d\xi)\right) dyds
   \\ &=:I_1+I_2,
 \end{align*}
where $\talpha=t-\e^{\alpha}$, with $\e \in (0,\frac t2]$ and $\alpha \in (0,1]$.  As in the proof of Proposition \ref{F^-1},  we can write
 \begin{equation}  \label{e40}
    \P\left(D_{w_{R,t}G_{R,t}}< \e\right)
    \leq \P\left(I_1 <2 \e\right)+\P\left( |I_2|> \e\right).
\end{equation}
We now estimate these probabilities in two steps.

\medskip
\noindent
{\it Step 1:} 
By Chebyshev inequality, for any $q\geq 2$,
\begin{equation} \label{e33}
 \P\left(I_1 <2 \e\right) \le  \P\left(I_1 ^{-1}> \frac 1{2 \e}\right)  \le (2\e )^q   \ex{ \left(\int_{\talpha} ^t   \int_\R  \varphi_{t,R}^2(s,y) U^2(s,y)dyds \right)^{-q} }.
 \end{equation}
Set
\[
m(\e, R)=   \int_{\talpha} ^t   \int_\R  \varphi_{t,R}^2(s,y)dyds.
\]
Using Lemma \ref{phivarphi} part (b),  taking into account that $s>\frac t2$, for all $R\ge R_0$, we have
 \begin{align} \label{lowerm}
m(\e, R)  \geq \frac{c_t\e^{\alpha}}{\log R}.
\end{align}
Then, because the function $x\to x^{-q}$ is convex,  applying Jensen's inequality, we can write
\begin{align}  \nonumber
& \ex{ \left(\int_{\talpha} ^t   \int_\R  \varphi_{t,R}^2(s,y) U^2(s,y)dyds \right)^{-q} }\\
& \le   {m(\e, R)^{-q-1} }     \int_{\talpha} ^t   \int_\R     \varphi_{R,t}^2(s,y) 
\ex{U^{-2q}(s,y)}dyds.  \label{e31}
\end{align}
Since $\{U(s,y) : y\in \R\}$ is stationary, we have  for all $s\in [\frac t2,2]$
\begin{align}  \nonumber
   \ex{\left(U(s,y)\right)^{-2q}} &=\ex{\left(U(s,0)\right)^{-2q}}= (p_s(0))^{2q}  \ex{\left(u(s,0)\right)^{-2q}} \\
   & \le (\pi t)^{-q}  \ex{\left( \inf_{s\in [\frac t2 ,t] } u(s,0)\right)^{-2q}}  =c_{t,q} <\infty,  \label{e30}
   \end{align}
   where  $c_{t,q} $ is a constant depending on $q$ and $t$ and the last equality follows from \cite[Theorem  1.4]{ChKi}.
   In what follows,  $c_{t,q} $ will denote a generic  constant depending on $q$ and $t$.
   Substituting \eqref{e30} into \eqref{e31} and using Lemma   \ref{phivarphi}  part (b)  and \eqref{lowerm} yields
   \begin{align}  \nonumber
    \ex{ \left(\int_{\talpha} ^t   \int_\R  \varphi_{t,R}^2(s,y) U^2(s,y)dyds \right)^{-q} } &
\leq  c_{t,q}{m(\e, R)^{-q-1} }   \int_{\talpha} ^t     \int_{\R}     \varphi_{R,t}^2(s,y)  dy
ds  \\  \nonumber
& \leq  c_{t,q}{m(\e, R)^{-q-1} } \int_{\talpha} ^t      \frac 1{s \log R} ds\\
& \le c_{t,q} \e^{-\alpha q}  (\log R)^{q},  \label{e32}
\end{align}
for $R\ge R_0$.
Finally, from \eqref{e33} and \eqref{e32}, we get
\begin{equation} \label{e39}
 \P\left(I_1 <2 \e\right)  \leq c_{t,q} \left(\log R\right)^{q}\e^{q(1-\alpha)}.
\end{equation}

\medskip
\noindent {\it  Step 2:}  Set $\Pi=  \P\left( |I_2|> \e\right)$. Using Fubini's theorem and Chebyschev's inequality for any $q\geq2$, we have
\begin{align*}
\Pi \leq \frac{1}{\e^q}\ex{\left|\int_{[\talpha,t]\times\R}  \left(\int_{\R}\int_{\talpha}^{\tau}   \varphi_{R,t}(\tau,\xi) \varphi_{R,t}(s,y)U(s,y)D_{s,y}U(\tau,\xi)dsdy\right)W(d\tau,d\xi) \right|^q}.
\end{align*}
Then, applying Burkholder-Davis-Gundy inequality, followed by Minkowski's inequality, we get for any $q \geq 2$
\begin{align}   \nonumber
    \Pi &\leq \frac{c_q}{\e^q}\ex{\left(\int_{\talpha}^t\int_{\R}\left(\int_{\talpha}^{\tau}\int_{\R}  \varphi_{R,t}(\tau,\xi)\varphi_{R,t}(s,y)U(s,y)D_{s,y}U(\tau,\xi)dyds\right)^2  d\xi d\tau \right)^{\frac q2}}
    \\ &=\frac{c_q}{\e^q}\E \Bigg[\Bigg(\int_{\talpha}^t\int_{\talpha}^{\tau} \int_{\talpha}^{\tau}\int_{\R^3}  \varphi^2_{R,t}(\tau,\xi)\varphi_{R,t}(s_1,y_1)\varphi_{R,t}(s_2, y_2)  \nonumber \\  \nonumber
    & \qquad \qquad \qquad \qquad  \qquad \qquad \qquad  \times  Y_{s_1,y_1, s_2, y_2}(\tau,\xi)dy_1dy_2  d\xi ds_1ds_2 d\tau \Bigg)^{\frac q2}  \Bigg] \nonumber
    \\ & \leq \frac{c_q}{\e^q}\Bigg(\int_{\talpha}^t\int_{\talpha}^{\tau} \int_{\talpha}^{\tau} \int_{\R^3}  \varphi^2_{R,t}(\tau,\xi)\varphi_{R,t}(s_1,y_1)\varphi_{R,t}(s_2, y_2)   \nonumber \\  
    & \qquad \qquad \qquad \qquad  \qquad \qquad \qquad  \times \|Y_{s_1,y_1, s_2, y_2}(\tau,\xi) \|_{q/2}  dy_1dy_2  d\xi ds_1ds_2 d\tau \Bigg)^{\frac q2},  \label{e35}
\end{align}
where 
\begin{align*}
    Y_{s_1,y_1,s_2,y_2} (\tau,\xi):=&U(s_1,y_1)D_{s_1,y_1}U(\tau,\xi)U(s_2,y_2)D_{s_2,y_2}U(\tau,\xi).
\end{align*}
Note that using  the estimates  \eqref{estU} and \eqref{estDU} and H\"older's inequality, we can write
\begin{equation} \label{e34}
    \left\| Y_{s_1,y_1,s_2,y_2}( \tau,\xi)\right\|_{q/2} \leq c_{t,q}p_{\frac{s_1(\tau-s_1)}{\tau}}(y_1-\frac{s_1}{\tau}\xi)p_{\frac{s_2(\tau-s_2)}{\tau}}(y_2-\frac{s_2}{\tau}\xi).
\end{equation}
Substituting the estimate  \eqref{e34} into \eqref{e35}, we obtain
\begin{equation} \label{e37}
\Pi \le  \frac {c_{t,q}} {\e^q} 
\left( \int_{\talpha}^t  \int_{\R}   \varphi^2_{R,t}(\tau,\xi) \left( \int_{\talpha}^{\tau}\int_{\R}  \varphi_{R,t}(s,y)  p_{\frac{s(\tau-s)}{\tau}}(y-\frac{s}{\tau}\xi) dyds \right)^2
d\xi d\tau \right)^{q/2}.
\end{equation}
Using the semigroup property, we have 
\begin{align*}
    &\int_{\R}\varphi_{R,t}(s,y)p_{\frac{s(\tau-s)}{\tau}}(y-\frac{s}{\tau}\xi)dy\leq \frac{1}{\Sigma_{R,t}}\int_{\R^2} p_{\frac{s(t-s)}{t}}(y-\frac{s}{t}x) p_{\frac{s(\tau-s)}{\tau}}(y-\frac{s}{\tau}\xi)dydx
    \\&=\frac{1}{\Sigma_{R,t}}\int_{\R} p_{\frac{s(t-s)}{t}+\frac{s(\tau-s)}{\tau}}(\frac{s}{t}x-\frac{s}{\tau}\xi)dx= \frac{t}{s\Sigma_{R,t}}\int_{\R} p_{\frac{t(t-s)}{s}+\frac{t^2(\tau-s)}{s\tau}}(x-\frac{t}{\tau}\xi)dx=  \frac{t}{s\Sigma_{R,t}},
\end{align*}
where we used  the identity $p_t(ax)=\frac{1}{a}p_{t/a^2}(x)$. Hence,  taking into account that $t_\alpha  >\frac t2$, we can write
\begin{equation} \label{e36}
   \int_{\talpha}^{\tau} \int_{\R}\varphi_{R,t}(s,y)p_{\frac{s(\tau-s)}{\tau}}(y-\frac{s}{\tau}\xi)dyds\leq \frac{1}{\Sigma_{R,t}}  \int_{\talpha}^{\tau}   \frac ts ds\le \frac {2 \e^\alpha} {\Sigma_{R,t}}.
\end{equation}
Finally, plugging the estimate  \eqref{e36}  into  \eqref{e37}, and using Lemma \ref{variances} part (b)  and Lemma \ref{phivarphi} part (b), we get for $R\ge R_0$,
 \begin{align}  \nonumber
    \Pi &\leq  c_{t,q}  \e^{q(\alpha-1)}  (R\log R)^{-q/2} \left( \int_{\talpha}^t\int_{\R}\varphi^2_{R,t}(\tau,\xi)  d\xi d\tau\right)^{q/2} \\ 
 &   \leq   c_{t,q}  R^{-q/2}\e^{(\frac{3\alpha}{2}-1)q}.  \label{e38}
\end{align}
Now,  choosing  $\alpha=4/5$,   we get, substituting \eqref{e38} and \eqref{e39} into \eqref{e40},
\begin{align*}
\P\left(D_{w_{R,t}}G_{R,t} < \e \right) \leq c_{t,q} (\log R)^{q}\e^{q/5}.
\end{align*}
Using this estimate, we get 
\begin{align*}
 \ex{\left( D_{w_{R,t}}G_{R,t} \right)^{-p} } &=p\int_0^{\infty} \e^{-p-1} \P\left(D_{w_{R,t}}G_{R,t} <\e\right)d\e  \\
 &\leq 1+ p \int_0^{1} \e^{-p-1} \P\left(D_{w_{R,t}}G_{R,t} <\e\right)d\e  \\
 &\leq  1+c_{t,q} (\log R)^{q}p \int_0^{1} \e^{-p-1+q/5} d\e .
\end{align*}
Finally, for $q=\gamma p>5p$, and for $R\ge R_0$, we obtain
 \begin{align*}
   \left\| \left(D_{w_{R,t}}G_{R,t}\ \right)^{-1}\right\|_p \leq c_{t,p,\gamma} (\log R)^{ \gamma},
\end{align*}
which completes our proof.

\end{proof}

Now, we are ready to prove Theorem \ref{densitypam}.

\begin{proof}[Proof of Theorem \ref{densitypam}]
We will apply Theorem \ref{densityapprox2} to the random variable  $G_{R,t} = \delta (w_{R,t})$. 
  Proposition \ref{G^-1} provides the estimate 
  \begin{equation} 
  \label{e41}
   \left\|\left(D_{w_{R,t}}G_{R,t} \right)^{-1}\right\|_4\leq c_{t,4,\gamma}(\log R)^{\gamma},
   \end{equation}
   for any $\gamma>5$, and for $R$ large enough.
 Moreover, from the proof of Theorem 1.1 in \cite{ChKhNuPu2},  we have 
 \begin{equation} \label{e42}
\left\|1-D_{w_{R,t}}G_{R,t}\right\|_{2}\leq \frac{C_t\sqrt{ \log R} }{\sqrt{R}}.
\end{equation}
We are only left to estimate the term $\left\| D_{w_{R,t}}\left({D_{w_{R,t}}G_{R,t}}\right)\right\|_{2}$.  
Recall that from \eqref{DwG} we have 
 \begin{align*}
   D_{w_{R,t}} G_{R,t}&=\frac{1}{\Sigma_{R,t}} \int_{0}^t\int_{\R}\int_{Q_R}\varphi_{R,t}(s,y) U(s,y) D_{s,y}U(t,x)dxdyds.
     \end{align*}
Applying again the derivative operator, we obtain
\begin{align*}
D_{r,z}  \left(D _{w_{R,t}}G_{R,t}\right) & =\frac{1}{\Sigma_{R,t}} \int_{0}^t\int_{\R}\int_{Q_R} \varphi_{R,t}(s,y) \Big(D_{r,z}U(s,y) D_{s,y}U(t,x) \\
& \quad +U(s,y)D_{s,y}D_{r,z}U(t,x)\Big)dxdyds,
\end{align*}
so that,
\begin{align*}
D_{w_{R,t}}\left({D_{w_{R,t}}G_{R,t}}\right)=\frac{1}{\Sigma_{R,t}}& \int_{0< r <s<t}\int_{\R^2}\int_{Q_R} dxdy dz  ds dr \varphi_{R,t}(s,y)\varphi_{R,t}(r,z) U(r,z) \\
 & \quad  \times \left(D_{r,z}U(s,y) D_{s,y}U(t,x) +2U(s,y)D_{r,z}D_{s,y}U(t,x)\right).
\end{align*}
Now using \eqref{DU} and \eqref{DDU} for $D_{s,y}U(t,x)$ and $D_{r,z}D_{s,y}U(t,x)$, we get 
\begin{align*}
D_{w_{R,t}}\left({D_{w_{R,t}}G_{R,t}}\right)=2\x^1_{R,t}+\x^2_{R,t}+2\x^3_{R,t},
\end{align*}
where: \begin{align*}
    \x^1_{R,t}&=  \int_{0}^t\int_{r}^t\int_{\R^2}dzdy ds dr \varphi_{R,t}^2(s,y)\varphi_{R,t}(r,z) U(r,z)U(s,y) D_{r,z}U(s,y), \\
    & \\
    \x^2_{R,t}&=\int_{0}^t\int_{r}^t\int_{\R^2}dzdy dsdr \varphi_{R,t}(s,y)\varphi_{R,t}(r,z) U(r,z) D_{r,z}U(s,y)  \\ 
    & \hspace{4cm} \times \int_{(s,t)\times\R}\varphi_{R,t}(\tau,\xi) D_{s,y}U(\tau,\xi) W(d\tau,d\xi), \\
    \x^3_{R,t}&=\int_{0}^t\int_{r}^t\int_{\R^2}dzdy ds dr \varphi_{R,t}(s,y)\varphi_{R,t}(r,z) U(r,z) U(s,y) 
     \\ 
    & \hspace{4cm} \times \int_{( s,t)\times\R}\varphi_{R,t}(\tau,\xi) D_{r,z}D_{s,y}U(\tau,\xi) W(d\tau,d\xi).
\end{align*}
As a consequence, we have
 \begin{equation}  \label{eq2}
    \left\|D_{w_{R,t}}\left(D_{w_{R,t}}G_{R,t}\right) \right\|_2 \leq 2\left\|\x^1_{R,t}\right\|_2+\left\|\x^2_{R,t} +2 \x^3_{R,t}\right\|_2.
    \end{equation}
We will further estimate the two terms in the right-hand side of the previous display.

\medskip
\noindent
 {\it Estimation of  } $\left\|\x^1_{R,t}\right\|_2$:
    Using the estimates \eqref{estU} and \eqref{estDU} and applying H\"older's inequality, we can write
\[
       \left\|U(s,y)U(r,z) D_{r,z}U(s,y)\right\|_2\leq    C_t p_{\frac{r(s-r)}{s}}(z-\frac{r}{s}y).
\]
    Therefore,
    \begin{align}  \nonumber
       \left\|\x^1_{R,t}\right\|_2 &\leq  \int_0^t\int_{r}^t \int_{\R^2}dzdy ds dr \varphi_{R,t}^2(s,y)\varphi_{R,t}(r,z)\left\| U(s,y)U(r,z) D_{r,z}U(s,y) \right\|_2 \\  \label{e50}
         &\leq C_t  \int_0^t\int_{r}^t\int_{\R^2} \varphi^2_{R,t}(s,y)\varphi_{R,t}(r,z)p_{\frac{r(s-r)}{s}}(z-\frac{r}{s}y) dz dy ds dr=:I_1.
    \end{align}  
    To estimate $I_1$,  we first integrate in $z$ and use the semigroup property, to obtain
     \begin{align} \nonumber
   \int_{\R} \varphi_{R,t}(r,z) p_{\frac{r(s-r)}{s}}(z-\frac{r}{s}y)dz &=\frac{1}{\Sigma_{R,t}} \int_{Q_R}\int_{\R} p_{\frac{r(t-r)}{t}}(z-\frac{r}{t} x)
 p_{\frac{r(s-r)}{s}}(z-\frac{r}{s}y)dz dx  \\&= \nonumber
   \frac{1}{\Sigma_{R,t}} \int_{Q_R} p_{\frac{r(t-r)}{t}+\frac{r(s-r)}{s}}(\frac{r}{s}y-\frac{r}{t}x) dx  \\
   &=   \frac{s}{r \Sigma_{R,t}} \int_{Q_R} p_{\frac{s^2(t-r)}{tr}+\frac{s(s-r)}{r}}(y-\frac{s}{t}x) dx . \label{e49}
\end{align}
 Now using the estimate $\varphi_{R,t}(s,y) \leq \frac{t}{\Sigma_{R,t} s}$ for one of the factors together with \eqref{e49} and then applying the semigroup property in $y$, we get
  \begin{align}  \nonumber
  &\int_{\R^2} \varphi^2_{R,t}(s,y)\varphi_{R,t}(r,z)p_{\frac{r(s-r)}{s}}(z-\frac{r}{s}y)dz dy \\  \nonumber
  &\leq \frac{t}{r\Sigma_R^3 } \int_{Q^2_R} \int_{\R} p_{\frac{s(t-s)}{t}}(y-\frac{s}{t}x_1)p_{\frac{s^2(t-r)}{tr}+\frac{s(s-r)}{r}}(y-\frac{s}{t}x_2)dy dx_1 dx_2 \\  \nonumber
     &= \frac{t}{r\Sigma_R^3  } \int_{Q^2_R}  p_{\frac{s(t-s)}{t}+\frac{s^2(t-r)}{tr}+\frac{s(s-r)}{r}}(\frac{s}{t}(x_1-x_2))dx_1 dx_2 \\  \nonumber
     &= \frac{t^2}{sr \Sigma_R^3  } \int_{Q^2_R}  p_{\frac{t(t-s)}{s}+\frac{t(t-r)}{r}+\frac{t^2(s-r)}{sr}}(x_1-x_2)dx_1 dx_2  \\ \nonumber
     &= \frac{t^2}{sr \Sigma_R^3   } \int_{Q^2_R}  p_{\frac{2t(t-r)}{r}}(x_1-x_2)dx_1 dx_2   \\ \label{e51}
          &= \frac{4R t^2}{\pi sr \Sigma_R^3  } \int_{\R}   \varphi(\xi) e^{-\frac{2t(t-r)}{r R^2}\xi^2}d\xi, 
 \end{align}
 where  the last equality follows from  Lemma \ref{xi}.   
    So,   substituting \eqref{e51} into  \eqref{e50}, we get
\[
        I_1 \leq C_t \frac{R}{\Sigma_{R,t}^3}  \int_{\R}   \varphi(\xi)   \int_0^t  \frac 1s \int_{0}^s\frac{1}{  r } e^{-\frac{2s(s-r)}{r}  \frac {\xi^2} {R^2}} dr ds d\xi.
        \]
        By  Lemma \ref{lem1}, we can write
\[
        I _1\leq C_t \frac{R\log R }{\Sigma_R^3} \left(  \int_{\R}   \varphi(\xi)  \log(e+ \frac 1 { \sqrt{2} |\xi|} ) d\xi  \right) \left(\int_0^t   \log(e+ \frac 1s)  ds\right).
        \]        
        Finally Lemma \ref{variances} part (b) yields
        \begin{equation} \label{eq3}
        I _1\le C_t (R \log R)^{-1/2}.
        \end{equation}
      
   \medskip
   \noindent
 {\it Estimation of }  $\left\|\x^2_{R,t} +2 \x^3_{R,t}\right\|_2$:
 Define
 \[
 V_{r,z, s, y}(\tau ,\xi)=U( r, z) D_{r, z} U(s,y) D_{s,y}u(\tau,\xi) + 2U( r, z)  U(r,z)D_{r,z} D_{s,y} U(\tau,\xi).
 \]
 With this notation in mind, we can write
\[
 \x^2_{R,t} +2 \x^3_{R,t}  =  \int_{[0,t]\times \R}    
  \left( \int_0^{\tau}  \int_{r} ^\tau \int _{\R^2}  \varphi_{R,t} (s,y)  \varphi_{R,t} ( r, z) V_{r,z, s,y}(\tau ,\xi )  ds dr dy dz  \right) \varphi_{R,t} (\tau,\xi) W(d\tau,d\xi).
\]
 Using  the Ito-Walsh isometry of the stochastic integral and Cauchy-Schwarz inequality, we obtain
 \begin{align*}
 I_2 &=:   \|  \x^2_{R,t} +2 \x^3_{R,t} \|^2_2\\
  & = \int_0^t \int_{\R}    \ex{ \left( \int_0^{\tau}  \int_{r} ^\tau \int _{\R^2}  \varphi_{R,t} (s,y)  \varphi_{R,t} ( r, z) V_{r,z,s, y}(\tau,\xi)  ds dr dy dz  \right) ^2} \\
  &  \qquad \times \varphi^2_{R,t} (\tau,\xi) d\xi d\tau \\
  &= \int_0^t \int_{\R}      \int_{\substack{0\leq r_1 \leq  s_1\leq \tau \\0\leq r_2 \leq s_2 \leq \tau}}\int_{\R^4}  \prod_{i=1,2}dy_i dz_idr_i ds_i \varphi_{R,t}(s_i,y_i)\varphi_{R,t}(r_i,z_i)   \\
  & \qquad \times \|   V_{r_i,z_i, s_i, y_i}(\tau,\xi)\| _2 \varphi^2_{R,t} (\tau,\xi) d\xi d\tau .
  \end{align*}
Using \eqref{estU}  \eqref{estDU} and \eqref{estDDU}, we see that, for $i=1,2$,
\begin{align*}
 \|   V_{r_i,z_i, s_i, y_i}(\tau,\xi)\| _2  \leq & C_tp_{\frac{s_i(\tau-s_i)}{\tau}}(y_i-\frac{s_i}{\tau }\xi )p_{\frac{r_i(s_i-r_i)}{s_i}}(z_i-\frac{r_i}{s_i}y_i) 
   \end{align*}
and hence
    \begin{align} \nonumber
  I_2  &\leq  C_t\int_0^{t}\int_0^{t}\int_0^{s_1}\int_0^{s_2}\int_{s_1\vee s_2}^t \int_{\R}d\xi  d\tau dr_1dr_2ds_1ds_2 \varphi^2_{R,t}(\tau,\xi)  \\   \label{e61}
  &\qquad \times \prod_{i=1,2}\int_{\R^2} 
  \varphi_{R,t}(s_i,y_i)\varphi_{R,t}(r_i,z_i) p_{\frac{s_i(\tau-s_i)}{\tau}}(y_i-\frac{s_i}{\tau }\xi )p_{\frac{r_i(s_i-r_i)}{s_i}}(z_i-\frac{r_i}{s_i}y_i) dz_i dy _i  .
  \end{align}
Integrating in the variable $z_i$ and using the semigroup property, we have 
\begin{align}  \nonumber
    &\int_{\R}\varphi_{R,t}(r_i,z_i) p_{\frac{r_i(s_i-r_i)}{s_i}}(z_i-\frac{r_i}{s_i}y_i) dz_i
   \\ \nonumber
   &=\frac{1}{\Sigma_{R,t}} \int_{Q_R}\int_{\R} p_{\frac{r_i(t-r_i)}{t}}(z_i-\frac{r_i}{t}x_i) p_{\frac{r_i(s_i-r_i)}{s_i}}(z_i-\frac{r_i}{s_i}y_i)dz_idx_i
   \\ \nonumber
   &=\frac{1}{\Sigma_{R,t}} \int_{Q_R}  p_{\frac{r_i(t-r_i)}{t}+\frac{r_i(s_i-r_i)}{s_i} }(\frac{r_i}{t}x_i-\frac{r_i}{s_i}y_i) dx_i 
   \\  \label{e60}
   &=\frac{s_i}{r_i\Sigma_{R,t}} \int_{Q_R}  p_{\frac{s_i^2(t-r_i)}{r_it}+\frac{s_i(s_i-r_i)}{r_i} }(\frac{s_i}{t}x_i-y_i) dx_i. 
\end{align}
From \eqref{e60}, using the estimate $\varphi_{R,t}(s_i,y_i) \leq \frac{t}{s_i\Sigma_{R,t}} $, and applying the semigroup property, we see that
\begin{align} \nonumber
   & \int_{\R^2}\varphi_{R,t}(s_i,y_i)\varphi_{R,t}(r_i,z_i)  p_{\frac{s_i(\tau-s_i)}{\tau}}(y_i-\frac{s_i}{\tau }\xi )p_{\frac{r_i(s_i-r_i)}{s_i}}(z_i-\frac{r_i}{s_i}y_i) dz_i dy _i \\ \nonumber
    &\leq  \frac{t}{\Sigma_{R,t}^2 r_i}  \int_{Q_R}\int_{\R}  p_{\frac{s_i^2(t-r_i)}{r_it}+\frac{s_i(s_i-r_i)}{r_i} }(\frac{s_i}{t}x_i-y_i)   p_{\frac{s_i(\tau-s_i)}{\tau}}(y_i-\frac{s_i}{\tau }\xi ) dy_i dx_i\\ \nonumber
   &=\frac{t}{\Sigma_{R,t}^2 r_i}  \int_{Q_R} p_{\frac{s_i^2(t-r_i)}{r_it}+\frac{s_i(s_i-r_i)}{r_i}+\frac{s_i(\tau-s_i)}{\tau}}(\frac{s_i}{t}x_i-\frac{s_i}{\tau}\xi) dx_i 
   \\  \label{e62}
   &=\frac{t\tau}{\Sigma_{R,t}^2 r_is_i}  \int_{Q_R} p_{\frac{\tau^2(t-r_i)}{r_it}+\frac{\tau^2(s_i-r_i)}{r_is_i} +\frac{\tau(\tau-s_i)}{s_i}}(\frac{\tau}{t}x_i-\xi) dx_i. 
\end{align}
Substituting the estimate \eqref{e62} into \eqref{e61}, together with  bound $\varphi_{R,t}(\tau,\xi) \leq \frac{t}{\tau\Sigma_{R,t}} $, and then integrating in $\xi$ this time, we get
 \begin{align}  \nonumber
  &   \int_{\R}
   \varphi^2_{R,t}(\tau,\xi)    \prod_{i=1,2}\int_{\R^2} 
  \varphi_{R,t}(s_i,y_i)\varphi_{R,t}(r_i,z_i) p_{\frac{s_i(\tau-s_i)}{\tau}}(y_i-\frac{s_i}{\tau }\xi )p_{\frac{r_i(s_i-r_i)}{s_i}}(z_i-\frac{r_i}{s_i}y_i) dz_i dy _i  d\xi 
   \\  \nonumber
   &  \leq \frac{t^4}{\Sigma_{R,t}^6 r_1r_2s_1s_2}  \int_{Q_R^2} \int_{\R}  \prod_{i=1,2}p_{\frac{\tau^2(t-r_i)}{r_it}+\frac{\tau^2(s_i-r_i)}{r_is_i} +\frac{\tau(\tau-s_i)}{s_i}}(\frac{\tau}{t}x_i-\xi) dx_i  d\xi
   \\ \nonumber
   &=\frac{t ^4}{\Sigma_{R,t}^6 r_1r_2s_1s_2}  \int_{Q_R^2}  p_{\frac{\tau^2(t-r_1)}{r_1t}+\frac{\tau^2(s_1-r_1)}{r_1s_1} +\frac{\tau(\tau-s_1)}{s_1}
   +\frac{\tau^2(t-r_2)}{r_2t}+\frac{\tau^2(s_2-r_2)}{r_2s_2} +\frac{\tau(\tau-s_2)}{s_2}}
   (\frac{\tau}{t}(x_1-x_2)) dx_1dx_2 
   \\ \nonumber
   &=\frac{t^5}{\Sigma_{R,t}^6 \tau  r_1r_2s_1s_2}  \int_{Q_R^2}  p_{\frac{t(t-r_1)}{r_1}+\frac{t^2(s_1-r_1)}{r_1s_1} +\frac{t^2(\tau-s_1)}{\tau s_1}+\frac{t(t-r_2)}{r_2}+\frac{t^2(s_2-r_2)}{r_2s_2} +\frac{t^2(\tau-s_2)}{\tau s_2}}(x_1-x_2) dx_1dx_2
   \\ \nonumber
   &=\frac{t^5}{\Sigma_{R,t}^6  \tau r_1r_2s_1s _2}  \int_{Q_R^2}  p_{2t(\frac{t}{r_1}+\frac{t}{r_2}-\frac{t}{\tau }-1)}(x_1-x_2) dx_1dx_2\\  \label{60e}
   &= \frac{4 t^5R}{\pi \Sigma_{R,t}^6 \tau r_1r_2s_1s_2 }  \int_{\R} \varphi(\xi) e^{-2t(\frac{t}{r_1}+\frac{t}{r_2}-\frac{t}{\tau}-1) \frac{\xi^2}{R^2}} d\xi,
 \end{align}
 where in the last inequality we have used Lemma  \eqref{xi}.
 Moreover,  using the bound 
 \begin{align*}
     \frac{t}{r_1}+     \frac{t}{r_2}-\frac{t}{\tau}-1 \geq   \frac{t-r_1}{2r_1}+\frac{t-r_2}{2r_2},
 \end{align*}
 and substituting \eqref{60e} into \eqref{e61}, we obtain
\begin{align*}
 I_2  & \le  \frac {C_tt^5 R}{\Sigma_{R,t}^6} \int_{\R}
  \int_0^t  \int_0^t \int_{r_1}^{t} \int_{r_2}^{t}     \int_{s_1 \vee s_2}^{\tau}  \frac {\varphi(\xi)}{ \tau r_1r_2s_1s_2}    e^{-t(\frac{t-r_1}{r_1}+\frac{t-r_2}{r_2}) \frac{\xi^2}{R^2}} d\tau ds_1 ds_2  dr_1dr_2d\xi\\
  & \le 
  \frac {C_tt^5 R}{\Sigma_{R,t}^6} \int_{\R} \varphi(\xi) d\xi \int_0^t  \frac { d\tau} {\tau}
 \left( \int_0^\tau    \frac 1s  \int_0^s   \frac 1r e^{-s (\frac {s-r}r) \frac {\xi^2}{R^2}} dr   ds  \right)^2.
 \end{align*}
 By Lemma \ref{lem1}, we get
 \[
 I_2 \le \frac {C_tt^5 R( \log R)^2}{\Sigma_{R,t}^6} \int_{\R} \varphi(\xi)  \int_0^t  \frac { d\tau} {\tau}
 \left( \int_0^\tau  
 \log(e+ \frac 1s) \log(e+ \frac 1 { |\xi|}) ds,
  \right)^2.
 \]
 which implies, in view of Lemma \ref{variances} part (b), 
 \begin{equation} \label{eq4}
  I_2 \le C_t \frac 1 { R^2 \log R}
  \end{equation}
  for all $R\ge R_0$.
  Plugging \eqref{eq3} and \eqref{eq4} into \eqref{eq2}, yields, for all $R\ge R_0$,
\begin{equation} \label{eq5}
\|D_{w_{R,t}}\left(D_{w_{R,t}} G_{R,t}\right) \|_2  \leq C_t  ( R \log R)^{-1/2} .
\end{equation}
Finally, from \eqref{e41}, \eqref{e42} and \eqref{eq5}, applying Theorem \ref{densityapprox2}
  we get
 \begin{align*}
   \sup_{x\in \R} |f_{G_R(t)}(x)-\phi(x)| \leq \frac{C_{t,\gamma}  (\log R) ^{2\gamma -\frac 12}}{\sqrt{R}},
   \end{align*}
 for all $R\ge R_0$,   which yields the desired estimate.
\end{proof}

\appendix 

\section{}

%\begin{lemma} Let $X$ be a real-valued nonnegative random variable. Then
  % $ \ex{X^{-p}}<\infty$ for all $p\geq 0$ if and only if   for all $q\geq 2$,  there is $C_q>0$ such that $\P(X<\e)< C_q \e^q$  for all $\e>0$.
%\end{lemma}

\begin{lemma} \label{KPhi}
     For $0<r<s<t$ and $y,z,x\in \R$,  we have 
     \begin{align*}
  K_{r,z,s,y}(t,x) \leq C_{t}  \Phi_{r,z,s,y}(t,x),  
\end{align*}
where $\Phi$ and $K$ are defined in \eqref{Phi} and \eqref{K} respectively.

\end{lemma}

\begin{proof}
Using the identity $p_t^2(a)=\frac{1}{\sqrt{2\pi t}}p_{t/2}(a)$, we see that the first term in $  K_{r,z,s,y}(t,x) $ is bounded by a constant depending on $t$ times the first term in $\Phi_{r,z,s,y}(t,x)$.  So, we estimate the integral term in $  K_{r,z,s,y}(t,x)$ that we denote by $I$.
Using the above identity for the square of the Gaussian together with the identity  
\eqref{identity}  we get
\begin{align*}
   I&= \int_{s}^t \int_{\R} p^2_{t-\theta}(x-w)p^2_{\theta-s}(w-y)p^2_{\theta-r}(w-z)\, dw d\theta  \\&= 
    \int_{s}^t\int_{\R}\frac{1}{\sqrt{(2\pi)^3(t-\theta)(\theta-s)(\theta-r)}} p_{\frac{t-\theta}{2}}(x-w)p_{\frac{\theta-s}{2}}(w-y)p_{\frac{\theta-r}{2}}(w-z)\,dw  d\theta \\&=
 p_{\frac{t-s}{2}}(x-y)   \int_{s}^t \int_{\R}\frac{1}{\sqrt{(2\pi)^3(t-\theta)(\theta-s)(\theta-r)}}  p_{\frac{(t-\theta)(\theta-s)}{2(t-s)}}( w-y- \frac { \theta-s}{ t-s}  (x-y))\\ 
 & \qquad  \times  p_{\frac{\theta-r}{2}}(w-z)\, dw d\theta .
    \end{align*}
  Now, applying the semigroup property, 
  \begin{align*}
      I=  \frac{p_{\frac{t-s}{2}}(x-y)}{(2\pi)^{3/2}}\int_{s}^t \frac{1}{\sqrt{(t-\theta)(\theta-s)(\theta-r)}} p_{\frac{(t-\theta)(\theta-s)}{2(t-s)}+\frac{\theta-r}{2}}(z-y-\frac{\theta-s}{t-s}(x-y))d\theta.
  \end{align*}
Since for $r<s<\theta<t$ 
\[
 \frac{\theta-r}{2} \leq\frac{(t-\theta)(\theta-s)}{2(t-s)}+\frac{\theta-r}{2}\leq \frac{t-r}{2},
 \]
  we have 
  \[
  p_{\frac{(t-\theta)(\theta-s)}{2(t-s)}+\frac{\theta-r}{2}}(z-y-\frac{\theta-s}{t-s}(x-y))
  \leq \frac{\sqrt{t-r}}{\sqrt{\theta-r}}p_{\frac{t-r}{2}}(z-y-\frac{\theta-s}{t-s}(x-y))
 \]
and
 \begin{align*}
I&\leq \frac{p_{\frac{t-s}{2}}(x-y)}{(2\pi)^{3/2}}\int_{s}^t \frac{\sqrt{t-r}}{\sqrt{(t-\theta)(\theta-r)^2(\theta-s)}} p_{\frac{t-r}{2}  }
 (z-y-\frac{\theta-s}{t-s}(x-y)) d\theta\\
  &\leq  \frac{\sqrt{t-r}p_{\frac{t-s}{2}}(x-y)}{(2\pi)^{3/2}} J,
 \end{align*}
where
  \begin{align*}
 J  & =\int_{s}^t \frac{\sqrt{t-r}}{ (\theta-r) \sqrt{(t-\theta)(\theta-s)}} p_{\frac{t-r}{2}  }
 (z-y-\frac{\theta-s}{t-s}(x-y)) d\theta.
  \end{align*}
  Making the change of variables $\frac { \theta-s}{t-s} =\gamma$ and putting  $\beta=\frac{{s-r}}{t-s}>0$ yields $\theta-r=\theta-s+s-r=(t-s)(\gamma+\beta)$ and
  \[
  J=\frac{1}{t-s} \int_0^1 (1-\gamma)^{ -\frac 12} \gamma^{-\frac 12}(\gamma+\beta)^{-1} p_{\frac{t-r}{2}} (z-y + \gamma(y-x)) d\gamma.
  \]
  We consider two cases:
  
  \noindent
  {\bf Case 1:}  If $z-y$ and $z-x$ have same sign, then
  \[
  p_{\frac{t-r}{2}} (z-y + \gamma(y-x))  \le p_{\frac{t-r}{2}} (z-y)+p_{\frac{t-r}{2}} (z-x).
\]
 
 \noindent
  {\bf Case 2:}  If $z-y$ and $z-x$ have different sign, suppose firstly that
  $z-y>0$ and $z-x=z-y+y-x<0$. Then,  $ 0<z-y<-(y-x)$; so $|z-y| <|y-x|$ and
 \[
  p_{\frac{t-r}{2}} (z-y + \gamma(y-x))  \le   \frac 1 {\sqrt{\pi (t-r)}} \mathbf{1}_{\{|y-x| > |z-y|\}}.
   \]
  Similarly, if $z-y<0$ and $z-x=z-y+y-x>0$, then $0>z-y>-(y-x)$, which implies $|z-y| <|y-x|$ and we end up with the same inequality.
 
  Finally,  noting that for $\beta=\frac{{s-r}}{t-s}>0$
\[
  \int_0^1 (1-\gamma)^{-1/2}\gamma^{-1/2} (\gamma+\beta)^{-1}d\gamma  = \frac{1}{\sqrt{\beta (\beta+1)}} =\frac{t-s}{\sqrt{(t-r)(s-r)}},
\]
  we get 
  \begin{align*}
    I&\leq  \frac{\sqrt{t-r}p_{\frac{t-s}{2}}(x-y)}{(2\pi)^{3/2}} J \\ 
    &\leq C_T \frac{p_{\frac{t-s}{2}}(x-y)}{\sqrt{s-r}} \left( p_{\frac{t-r}{2}} (z-y)+p_{\frac{t-r}{2}} (z-x)+  \mathbf{1}_{\{|y-x| > |z-y|\}}  \right) \\
  &\leq C'_T \frac{p^2_{t-s}(x-y)}{\sqrt{s-r}} \left( p^2_{t-r} (z-y)+p^2_{t-r} (z-x)+  \mathbf{1}_{\{|y-x| > |z-y|\}}  \right) ,
  \end{align*}
which then completes our proof by taking the square roots on both sides.
\end{proof}

\begin{lemma}\label{l1Phi} Let $\Phi$ be as in \eqref{Phi}.  
   For fixed $0<r<s<t$ and $x\in \R$,
   \begin{align*}
   \int_{\R^2}  \Phi_{r,z,s,y}(t,x)dydz \leq C_t \left(1+\frac{1}{(s-r)^{1/4}}\right).
   \end{align*}
\end{lemma}

\begin{proof} Fix $0<r<s<t$ and $x\in \R$, using the semigroup property and Gaussian integrals, we have 
\begin{align*}
 & \int_{\R^2}p_{t-s}(x-y)\left( p_{s-r}(y-z) + \frac {p_{t-r} (z-y)+p_{t-r} (z-x)+  \mathbf{1}_{\{|y-x| > |z-y|\}} } { (s-r)^{1/4}}  \right)dy dz \\
 & = 1+ \frac{1}{(s-r)^{1/4}} +\frac{1}{(s-r)^{1/4}}  \int_{\R^2 } p_{t-s}(x-y)   \mathbf{1}_{\{|y-x| > |z-y|\}}  dydz\\ &
 \le C_t \left(1+\frac{1}{(s-r)^{1/4}} \right).
\end{align*}
\end{proof}

\begin{lemma} \label{variances} Let $\sigma^2_{R,t}$ and $\Sigma^2_{R,t}$ be as defined in \eqref{FR} and \eqref{GR} respectively.  Then\begin{itemize}
\item[(a)] $\displaystyle\lim_{R\to \infty} \frac{\sigma_{R,t}^2}{R}=2\int_0^t \xi(s)ds$ where $\xi(s)=\ex{\left(\sigma(u(s,y))\right)^2}$.

\vskip 2pt
\item[(b)] $\displaystyle\lim_{R\to \infty} \frac{\Sigma_{R,t}^2}{R\log R}=2t$.
\end{itemize}

\end{lemma}

\begin{proof} See proposition 3.1 in \cite{HuNuVi} for part (a) and proposition 4.1 in \cite{HuNuVi} for part (b).

\end{proof}

\begin{lemma} \label{phivarphi} Fix $t>0$. Let $\phi_{R,t}$ and $\varphi_{R,t}$ be defined as in \eqref{phi} and \eqref{varphi}. Then, there exists $R_0\ge 1$, depending on $t$, such that for all $0<s<t$ and $R\ge R_0$: 
\begin{itemize}
    \item[(a)] $c_t  \leq \displaystyle\int_{\R}\phi^2_{R,t}(s,y)dy \leq C_t$,
    where the lower bound holds for $t/2<s<t$.
    \vskip 2pt
    \item[(b)]$\dfrac{c_t}{s\log R} \leq  \displaystyle\int_{\R}\varphi^2_{R,t}(s,y)dy \leq  \frac {C_t} {s \log R} $ where the lower bound holds for  $t/2<s<t$.
\end{itemize}

\end{lemma}
\begin{proof}
\vskip 2pt
(a) We start with the upper bound. Using the semigroup property, we see that
   \begin{align*}
      \int_{\R}\phi^2_{R,t}(s,y)dy &=  \frac{1}{\sigma_{R,t}^2}\int_{Q_R^2}\int_{\R}p_{t-s}(y-x_1)p_{t-s}(x_2-y)dydx_1dx_2\\
      &=       \frac{1}{\sigma_{R,t}^2}\int_{Q_R^2}p_{2(t-s)}(x_1-x_2)dx_1dx_2 \\
     &  \leq \frac{1}{\sigma_{R,t}^2}\int_{Q_R}\int_{\R}p_{2(t-s)}(x_1-x_2)dx_1dx_2=\frac{2R}{\sigma_{R,t}^2} \leq C_t,
   \end{align*} where the last  bound follows from Lemma \ref{variances} part (a).
  To see the lower bound, let $R\geq 1$,  and $t/2<s<t$. Then,  
    \begin{align*}
 \int_{\R}\phi^2_{R,t}(s,y)dy&=\frac{1}{\sigma_{R,t}^2}\int_{Q_R^2}p_{2(t-s)}(x_1-x_2)dx_1dx_2 \ge  \frac{1}{2\sigma_{R,t}^2}  \int_{Q^2_{R/\sqrt{2}}}   p_{2(t-s)} (y_1) dy_1dy_2 \\ & \geq    \frac{ R}{ \sqrt{2}\sigma_{R,t}^2} \int_{-1/\sqrt{2}} ^{1/\sqrt{2}}  p_{2(t-s)} (y) dy  \geq c_t ,
    \end{align*}
     where the last  bound follows from Lemma \ref{variances} part (a).

(b) Similarly,  using the semigroup property, we see that 
   \begin{align*}
      \int_{\R}\varphi^2_{R,t}(s,y)dy&=  \frac{1}{\Sigma_{R,t}^2}\int_{Q_R^2}\int_{\R}p_{\frac{s(t-s)}{t}}(y-\frac{s}{t}x_1)p_{\frac{s(t-s)}{t}}(y-\frac{s}{t}x_2)dydx_1dx_2\\
      &=       \frac{1}{\Sigma_{R,t}^2}\int_{Q_R^2}p_{\frac{2s(t-s)}{t}}(\frac{s}{t}(x_1-x_2))dx_1dx_2 \\
      &= \frac{t^2}{s^2\Sigma_{R,t}^2}\int_{Q_{sR/t}^2}p_{\frac{2s(t-s)}{t}}( y_1-y_2)dy_1dy_2 
     \\& \leq  \frac{2Rt}{s\Sigma_{R,t}^2} \le  \frac {C_t} {s\log R},
   \end{align*} 
   for all $R\ge R_0$, 
  where the last  bound follows from Lemma \ref{variances} part (b).
  To see the lower bound, let $t/2<s<t$. Then, assuming  $R\ge 1$,
    \begin{align*}
 \int_{\R}\varphi^2_{R,t}(s,y)dy&= \frac{t^2}{s^2\Sigma_{R,t}^2}\int_{Q_{sR/t}^2}p_{\frac{2s(t-s)}{t}}( y_1-y_2)dy_1dy_2 \\
 & \ge  \frac{ \sqrt{2} tR }{s\Sigma_{R,t}^2}\int_{Q_{ \frac {sR} {t\sqrt{2}}}}p_{\frac{2s(t-s)}{t}}( z) dz \\
 &  \ge  \frac{ \sqrt{2} tR }{s\Sigma_{R,t}^2}  \P\left(    |N| \le  \frac R2 \sqrt{\frac s { t(t-s)}} \right)     \\
 & \ge                    \frac{ \sqrt{2} tR }{s\Sigma_{R,t}^2}  \P\left(    |N| \le    \frac 1{2\sqrt{t}}\right) 
   \geq   \frac{c_t}{s\log R},
    \end{align*}
  where the last  bound follows from Lemma \ref{variances} part (b) and $N$ denotes a $N(0,1)$ random variable.
\end{proof}

\begin{lemma} \label{xi} For all $R,t>0$, 
\begin{align*}
\int_{Q_R^2} p_{t}(x_1-x_2)dx_1dx_2 =\frac{4R}{\pi} \int_{\R} \varphi(\xi ) e^{-t\frac{\xi^2}{R^2}}d\xi,
\end{align*}
where 
\begin{align*}
\varphi(\xi)=\frac{1-\cos \xi}{\xi^2}.
\end{align*}
\end{lemma}

\begin{proof}
See Appendix in \cite{ChKhNuPu1}.
\end{proof}

\begin{lemma} \label{lem1}  For all $R\ge e$ and all $s>0$,
 \[
\frac 1s \int_0^s   \frac 1r e^{-s(\frac {s-r}r) \frac {\xi^2}{R^2}} dr\le 7 \log R  \log(e+ \frac 1s) \log(e+  \frac 1 { |\xi|}).
 \]
 \end{lemma}
 
 \begin{proof}
 See  \cite[Lemma A.1]{ChKhNuPu1}.
 \end{proof}

%%%%%%%%%%%%%%%%%%%%%%%%%%%%%%%%%%%%%%%%%%%%%%%%%%%%%%%%%%%%%%%%%%%%%%%%%%%%%%%%%%%%%%%%%%%%%%%%%%%%%%%%%%%%%%%%%%%%%%%%%%%%%%%%%%%%%%%%%%%%%%%%%%%%%%%%%%%%%%%%%%%%%%%%%%%%%%%%%%%%%%%%%%%%%%%%%%%%%%%%%%%%%%%%%%%%%%%%%%%%%%%%%%%%%%%%%%%%%%%%%%%%%%%%%%%%%%%%%%%%%%%%%%%%%%%%%%%%

{}

\end{document}